%% file: main.tex
\documentclass[a4paper,UKenglish,thm-restate]{article}

\usepackage{    
	amsmath, amssymb, mathrsfs, enumerate, graphicx,
	xcolor, caption, multicol, multirow, float,
    mathtools, bbm, extarrows, tikz-cd,
    appendix
}
\usepackage{amsthm}
\usepackage{algpseudocode}
\usepackage{thmtools, thm-restate}
\usepackage{graphicx}
\usepackage{subcaption}
\usepackage{authblk}

\usepackage[margin=1in]{geometry}
\usepackage{hyperref}
\usepackage{framed}

\newtheorem{theorem}{Theorem}[section]
\newtheorem*{theorem*}{Theorem}
\newtheorem{lemma}[theorem]{Lemma}

\newtheorem{proposition}[theorem]{Proposition}
\newtheorem*{proposition*}{Proposition}
\newtheorem{corollary}[theorem]{Corollary}

\theoremstyle{definition}
\newtheorem{definition}[theorem]{Definition}
\newtheorem{remark}[theorem]{Remark}

\addtotheorempostheadhook[remark]{\phantomsection}
\addtotheorempostheadhook[definition]{\phantomsection}
\addtotheorempostheadhook[lemma]{\phantomsection}
\addtotheorempostheadhook[corollary]{\phantomsection}
\addtotheorempostheadhook[theorem]{\phantomsection}
\addtotheorempostheadhook[proposition]{\phantomsection}

\usepackage{cleveref}

\title{Discrimination of dynamic data via curvature sets}

\author[1]{Nadezhda Belova}
\author[1]{Maxwell Goldberg}
\author[1]{Facundo M\'emoli}
\author[1]{Sriram Raghunath}
\author[1]{Andrew Xie}

\affil[1]{Department of Mathematics, Rutgers University}

\input{preamble.tex}

\begin{document}
\maketitle

\begin{abstract}
    Techniques from topological data analysis (TDA) have proven effective in studying time-dependent data arising in dynamic systems, such as animal swarming behavior and spatiotemporal patterns in neuroscience. While early algorithms leveraged efficient updates to persistence diagrams for dynamic data, they struggled to distinguish behaviors that are isometric at each fixed time but differ qualitatively. This limitation was addressed by Kim and Mémoli, who introduced a spatiotemporal persistence framework for dynamic metric spaces, resulting in multiparameter persistence modules. However, these modules pose computational challenges.

    To address this, we build on insights from Gómez and Mémoli, who observed that the homology of Rips complexes over size $(2k+2)$ point subsets of a metric space—termed principal curvature sets—is both tractable and informative. We extend this idea to dynamic settings by introducing dynamic curvature-set persistent homology, applying the spatiotemporal framework of Kim and Mémoli to curvature sets. We prove that the resulting multiparameter persistence modules are interval-decomposable: in fact, they possess a stronger property we term antichain-decomposable. Utilizing this property, we present a new algorithm to efficiently compute the erosion distance $d_E$ (due to Patel) between arbitrary antichain-decomposable modules (including, but not limited to modules produced by our construction). Additionally, our construction is stable with respect to a generalized Gromov-Hausdorff distance between time-dependent datasets proposed by Kim and Mémoli. This enables a robust computational pipeline for distinguishing dynamic data, as demonstrated in experiments with the Boids model, where we successfully detect parameter changes.
\end{abstract}

\section{Introduction}
    Topological data analysis has proven a relevant and effective tool for analyzing complex and dynamic systems, particularly those in biology and neuroscience \cite{Songdechakraiwut2020, Ciocanel2021, Catanzaro2024}. A common thread throughout these approaches is the use of topological invariants (e.g. barcodes) to represent the system at individual slices in time, before analyzing the time series of invariants to detect similarities in behavior or the appearance of certain patterns/motifs.
    
    As elucidated in \cite{spatiotemporal}, information is lost by computing topological invariants pointwise in time, as there exist examples of dynamic data that are isometric at each fixed time step, but nevertheless posses qualitatively different behavior. An example of this is illustrated in \Cref{fig:dmsisometryexample}. To produce a more informative invariant, in \cite{spatiotemporal} Kim and M\'emoli construct the spatiotemporal persistent homology, which produces multiparameter modules indexed by both a time interval and a scale parameter. This results in increased discrimination power at the cost of producing complex multiparameter modules.
    \begin{figure}[h]
        \centering
        \includegraphics[width=0.667\textwidth]{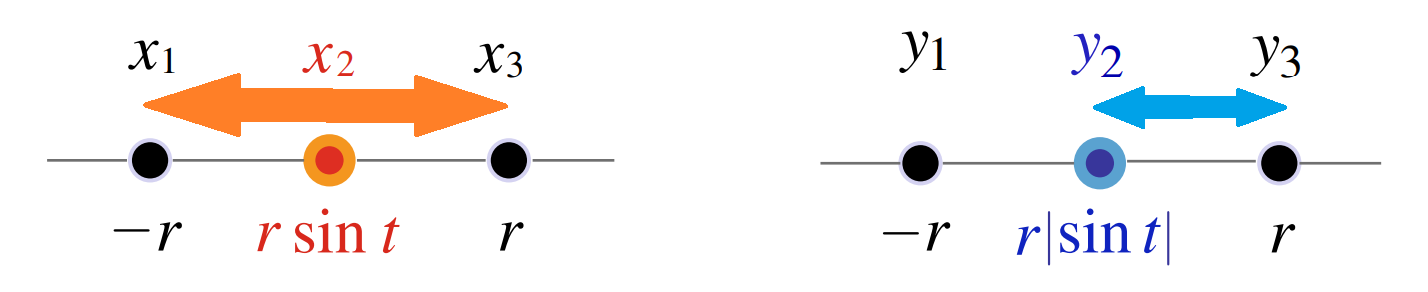}
        \caption{Two dynamic metric spaces exhibiting qualitatively different behavior. Here $r>0$, $x_1,x_3,y_1$ and $y_3$ are fixed at the displayed coordinate values whereas $x_2(t) = r\sin t$ and $y_2(t) = r |\sin t|$.  For any fixed $t$, the resulting configurations are isometric as metric spaces (when $x_2$ is to the left of the midpoint of $x_1$ and $x_3$, reflect the left space through the midpoint to show isometry). However, the two dynamic behaviors are certainly different. Figure courtesy of Kim and M\'emoli; see \cite{spatiotemporal}.}
        \label{fig:dmsisometryexample}
    \end{figure}

    In order to alleviate the computational cost of this multiparameter construction, we consider the idea of curvature sets, first introduced by Gromov in \cite{Gromov1999MetricSF}. This was utilized in G\'omez and M\'emoli in \cite{curvaturesets}, where it is shown that for a point cloud with sufficiently few points, the homology of the Rips filtration can be computed quickly in closed-form while still remaining informative. By considering all ``small enough'' subsets of arbitrary-size input data, a tractable invariant can be derived. 
    \subsection{Contributions}
    We propose a construction for topological analysis of dynamic data, utilizing the idea of curvature sets to provide a computational simplification of the spatiotemporal construction from \cite{spatiotemporal}. The resulting invariant is stable and computationally tractable.

    Furthermore, we show that this construction produces multidimensional persistence modules that \emph{antichain-decomposable}: that is, interval decomposable where any two distinct summands in the interval decomposition are, informally, ``incomparable''. These modules have additional properties such as thinness, i.e. pointwise dimension either $0$ or $1$, see \Cref{def:thinmodule}.

    We present a new algorithm to efficiently compute the erosion distance \cite{patelerosion} between any two antichain-decomposable persistence modules. We later generalize this algorithm to compute the erosion distance between any two modules with a generalized convexity property, see \ref{def:rankmaximal}. It is worth emphasizing that the erosion distance algorithm we present does not rely on the details of our particular construction.
    
    Finally, we demonstrate the tractability and effectiveness of our construction analyzing dynamic data produced by the boids model, introduced in \cite{boidsmodel}.
\section{Preliminaries}    
    \subsection{Dynamic Data Modeling}

    We follow \cite{stablesignatures} and \cite{spatiotemporal} in using the following mathematical representation of dynamic data sets.

    \begin{definition}[Dynamic metric spaces (DMS) \cite{stablesignatures}]
    
            A \emph{dynamic metric space} $(X, d_X(\cdot))$ is a finite set $X$ together with a function $d_X:\mathbb{R} \times X \times X \to \mathbb{R}_{\geq 0}$ such that 
            
            \begin{itemize}
                \item[(1)] For each $t \in \mathbb{R}$, $(X, d_X(t))$ is a pseudo-metric space
                \item[(2)] For each pair $x, x' \in X$, the map $t \mapsto d(t)(x, x')$ is continuous
            \end{itemize}
        
    \end{definition}

\begin{definition}[Semimetric space]
            For a set $X$, a function $d_X: X \times X \rightarrow \mathbb{R}_{\geq 0}$ is a \emph{semimetric} if it satisfies:
            \begin{itemize}
                \item $d_X(x, x) = 0$ for all $x \in X$
                \item $d_X(x, x') = d_X(x', x)$ for all $x, x' \in X$
            \end{itemize}
            The pair $(X,d_X)$ will be referred to as a \emph{semimetric space}.
        \end{definition}
        
        \begin{definition}[\cite{spatiotemporal}]
            Given any DMS $\gamma_X = (X, d_X(\cdot))$ and an interval $I \subset \mathbb{R}$, the semimetric $d_X(I)(\cdot, \cdot)$ on $X$ is given by $d_X(I)(x, x') = \inf_{t \in I} d_X(t)(x, x')$ for any $x, x' \in X$. Let $\gamma_X(I) := (X, d_X(I))$ denote the associated semimetric space.
        \end{definition}

        In order to appropriately encode the information of dynamic data sets, our data representations must combine information over different times. We recall the following sequence of definitions from \cite{spatiotemporal} to define the spatiotemporal Rips filtration of a DMS, which is an invariant recording data features persisting through time and scale.

        Although the Rips complex and filtration are ordinarily defined on metric spaces, they can nonetheless be constructed for a semimetric space.
        \begin{definition}[Rips complex]
            Let $(X,\,d_X)$ be a semimetric space. For each $\delta\in\mathbb{R}_+$ we define the \emph{Rips complex} of $X$ at scale $\delta$, $\mathcal{R}_\delta(d_X)$, as an abstract simplicial complex on $X$ where for finite $\sigma\subseteq X$ we have $\sigma\in\mathcal{R}_\delta(d_X)$ if and only if $d_X(x,\,y)\leq\delta$ for all $x,\,y\in\sigma$.
        \end{definition}

        By $\bsimp$ we denote the category of abstract simplicial complexes with simplicial maps. Also note that we consider $\mathbb{R}$ as a poset category. By $\mathbf{Vec}$ we denote the category vector spaces and linear maps (over a fixed field $\mathbb{F}$). 

        \begin{definition}[Rips filtration]
            Let $(X,\,d_X)$ be a semimetric space. The \emph{Rips filtration} of $(X,\,d_X)$ is the functor $\mathcal{R}_\bullet(d_X):\mathbb{R}_+ \to \bsimp$ defined by: To each $\delta \in \mathbb{R}_+$ assign $\mathcal{R}_\delta(d_X)$. Also to each morphism $\delta\leq\delta'\in\mathbb{R}$ assign the inclusion map $\mathcal{R}_\delta(d_X)\hookrightarrow\mathcal{R}_{\delta'}(d_X)$.
        \end{definition}

        Here by $\mathbf{Int}$ we denote the poset of compact intervals in $\mathbb{R}$ with inclusions. And $\mathbf{Int}\times\mathbb{R}$ is the product poset, also considered as a poset category. We denote this poset $\mathbf{Dyn}$.

        We are now prepared to define the spatiotemporal Rips filtration, which will be used later to construct our proposed invariant.

        \begin{definition}[Spatiotemporal Rips filtration \cite{spatiotemporal}]
            Given any DMS $\gamma_X=(X,\,d_X(\cdot))$, its \emph{spatiotemporal Rips filtration} is a functor $\mathcal{R}^{\mathrm{lev}}(\gamma_X):\mathbf{Dyn}\to\bsimp$ defined by: To each $(I,\,\delta)\in\mathbf{Dyn}$ assign $\mathcal{R}_\delta(\gamma_X(I))$. Also to each morphism $(I,\delta)\leq(I',\,\delta')\in\mathbf{Dyn}$ assign the inclusion map $\mathcal{R}_\delta(\gamma_X(I))\hookrightarrow\mathcal{R}_{\delta'}(\gamma_X(I'))$. 
        \end{definition}

        Again, following \cite{spatiotemporal} we define the $k$th dynamic Rips persistent homology module of a DMS $(X,\,\gamma_X)$ as \[H_k(\mathcal{R}^\mathrm{lev}(X,\,\gamma_X)): \bdyn \to \bvec\] where $H_k$ is the $k$th simplicial homology functor (over the fixed field $\mathbb{F}$). This is an example of a persistence module, a functor from a poset category into a category of modules. 

        \begin{remark}[Interleaving distance between $\bdyn$-modules] \label{rem:dyninterleavingdistance}
            The space of all $\bdyn$-indexed persistence modules is metrized by the interleaving distance $d_I^{\bvec}$. For a detailed exposition of this interleaving distance, see \cite[Section 4.1]{spatiotemporal}.
        \end{remark}

        Since the poset $\mathbf{Dyn}$ is multidimensional, standard results \cite{multiparamhard} imply that computational questions regarding interleaving distances between these persistence modules are in general intractable.

\subsection{Properties of Persistence Modules}

    A particularly important type of persistence module are the interval-decomposable persistence modules. We state the definitions over an arbitrary poset:

    \begin{definition}[Order-connectivity]
        Let $(\mathcal{P}, \leq)$ be a poset and $S \subset \mathcal{P}$. We say two elements $p, q \in S$ are \emph{order-connected} in $S$ if there exists a sequence $p = x_0, \ldots, x_n = q$ of elements in $S$, such that all adjacent elements are comparable: that is, for all $i \in \{0, \ldots, n-1\}$, either $x_i \geq x_{i+1}$ or $x_{i+1} \geq x_i$.
        
        Order-connectivity in $S$ is an equivalence relation on elements in $S$. Thus, we can partition $S$ into equivalence classes, i.e. $S = \bigsqcup_{\alpha \in A} C_{\alpha}$. We say each equivalence class $C_{\alpha}$ is an order-connected component of $S$.
    \end{definition}
    
    \begin{definition}[Interval] \label{def:interval}
        For a poset $(\mathcal{P}, \leq)$, an \emph{interval} of $\mathcal{P}$ is a subset $I \subset \mathcal{P}$ such that:
        \begin{enumerate}
            \item[(i)]{(Connectivity)} For any $p, q \in I$, the elements $p$ and $q$ are order-connected in $I$.
            \item[(ii)]{(Convexity)} For any $p, q \in I$ and $r \in \mathcal{P}$, if $p \leq r \leq q$ then $r \in I$.
        \end{enumerate}
        We will denote the collection of all intervals over $\mathcal{P}$ by $\mathbf{Int}(\mathcal{P})$.
    \end{definition}

    The following is a special case of intervals, that will be relevant in future sections:
    \begin{definition}[Segment] \label{def:segment}
        For a poset $(\mathcal{P}, \leq)$ and some $p, q \in \mathcal{P}$ satisfying $p \leq q$, the \emph{segment} $[p, q]$ is defined to be the set \[\{ x \in \mathcal{P} : p \leq x \leq q \}.\]
        We will denote the collection of all segments over $\mathcal{P}$ by $\mathbf{Seg}(\mathcal{P})$. This can be viewed as a poset category ordered by inclusion, which is isomorphic to the poset $\mathcal{P}^\text{op}\times\mathcal{P}$. This is because $[p_1, q_1] \subset [p_2, q_2]$ if and only if $p_1 \geq p_2$ and $q_1 \leq q_2$.
    \end{definition}

    Certain subsets of $\mathcal{P}$ can be used to construct persistence modules, in a natural way:
    \begin{definition}[Indicator Module]
        Let $\mathcal{P}$ be a poset, and $I$ be a convex (in the sense of property (ii) of \Cref{def:interval}) subset of $\mathcal{P}$. The \emph{indicator module} $\mathbb{I}_I: \mathcal{P} \rightarrow \bvec$ is the functor defined by
        \begin{center}
            $\mathbb{I}_I(p) \defeq \begin{cases} \mathbb{F} & p \in I\\ 0 & p \notin I \end{cases}
            \qquad
            \mathbb{I}_I(p \leq q) \defeq \begin{cases} \mathbf{Id}_\mathbb{F} & p, q \in I\\ 0 & \mathrm{otherwise} \end{cases}
            $
        \end{center}
        Note that functoriality follows from the fact that $I$ is convex.
    \end{definition}
    \begin{definition}[Interval Module] \label{def:intervalmodule}
        For a poset $(\mathcal{P}, \leq)$ we say a persistence module $\bbv: \mathcal{P} \rightarrow \bvec$ is an \emph{interval module} if $\bbv \cong \mathbb{I}_I$ for some interval $I$ of $\mathcal{P}$.
    \end{definition}
    
    \begin{definition}[Interval-Decomposable Module]
        Let $\mathcal{P}$ be a poset, and $\bbv: \mathcal{P} \rightarrow \bvec$ be a persistence module. We say $\bbv$ is \textit{interval-decomposable} if there exists a collection $\{I_\alpha\}_{\alpha \in A}$ of intervals over $\mathcal{P}$ such that \[\bbv \cong \bigoplus_{\alpha \in A} \mathbb{I}_{I_\alpha}.\]
    \end{definition}
    
    Another important class of modules is the following.
    \begin{definition} \label{def:thinmodule}
        Let $\mathcal{P}$ be a poset, and $\bbv: \mathcal{P} \rightarrow \bvec$ be a persistence module. We say $\bbv$ is a \emph{thin module} if $\dim \bbv(p) \in \{0, 1\}$ for all $p \in \mathcal{P}$
    \end{definition}

    Note that interval modules are thin, although the converse is in general not true. For a thin module, it is natural to make the following definition:
    \begin{definition}[Support of a thin module]
        The \emph{support} of a thin module $\bbv: \mathcal{P} \rightarrow \bvec$, denoted $\supp \bbv$, is the subset $\supp \bbv \defeq \{ p \in \mathcal{P} : \bbv(p) \cong \mathbb{F} \}$.
    \end{definition}

    We now define the main class of modules analyzed in this work:
    \begin{definition}[Antichain-decomposable modules] \label{def:acdmodules}
        Let $\mathcal{P}$ be a poset, and $\bbv: \mathcal{P} \rightarrow \bvec$ be a persistence module. We say $\bbv$ is \emph{antichain-decomposable} if there exists a family of intervals $\{S_\alpha\}_{\alpha \in A}$ such that $\bbv \cong \bigoplus_{\alpha \in A} \mathbb{I}_{S_\alpha}$ and the following holds:
        \[ p \in S_\alpha, q \in S_\beta \implies p \not \leq q \text{ and } q \not \leq p \qquad \forall \alpha \neq \beta. \]
        Informally, this means that all summands of the interval decomposition are ``incomparable''.
    \end{definition}
    For brevity, we will use the acronym ACD to refer to the antichain-decomposable property. 
    \begin{remark}[Properties of ACD modules] \label{rem:acdproperties}
        If $\bbv: \mathcal{P} \rightarrow \bvec$ is an ACD module, the following are true:
        \begin{itemize}
            \item $\bbv$ is a thin module. Suppose not: then there must be two summands $\mathbb{I}_S$ and $\mathbb{I}_T$ in the interval decomposition of $\bbv$ that satisfy $S \cap T \neq \emptyset$. But then there exists an $x$ such that $x \in S$ and $x \in T$, and of course $x \leq x$: but this contradicts the fact that $\bbv$ is ACD. So $\bbv$ must be thin.
            \item $\supp \bbv$ is a convex subset of $\mathcal{P}$. Suppose not: then there exists $p, r, q \in \mathcal{P}$ with $p \leq r \leq q$, such that $p, q \in \supp \bbv$ but $r \notin \supp \bbv$. Because intervals are convex, there must be two distinct summands $\mathbb{I}_S, \mathbb{I}_T$ in the interval decomposition of $\bbv$ such that $p \in S, q \in T$. But the fact that $p$ and $q$ are comparable contradicts the fact that $\bbv$ is an ACD module.
        \end{itemize}
        Note that the converse (thin module with convex support) does not necessarily imply ACD, such modules may in general not even be interval-decomposable.
    \end{remark}
    
    Finally, for computational purposes we introduce the notion of boundedness, when the indexing poset is $\mathbb{R}^d$ with the product order.
    \begin{definition}[Bounded ACD module] \label{def:boundedacdmodule}
        For an ACD module $\bbv: \mathbb{R}^d \rightarrow \bvec$, we say $\bbv$ is \emph{bounded} if $\supp(\bbv)$ is a bounded subset of $\mathbb{R}^d$.
    \end{definition}

\section{Curvature Set Simplification of Persistence Modules}
    To motivate our construction in the general case (for arbitrary-size dynamic data), we will first consider the Rips persistent homology of a static point cloud.
    
    Let $(X, d_X)$ be a finite metric space. If $X$ has exactly three points, i.e. $|X| = 3$, then $H_1(\mathcal{R}_\delta(d_X)) \cong \{ 0 \}$ for any $\delta \geq 0$: this is because $\mathcal{R}_\delta(d_X)$ is a flag complex, so the moment it contains a cycle with three edges, it also contains a filled in $2$-simplex, which makes $\mathcal{R}_\delta(d_X)$ contractible.
    
    However, something interesting happens with the degree-1 homology when $|X| = 4$. Consider the point clouds shown in \Cref{fig:fourpointmetricspaces}. 
    
    \input{four_point_spaces}
    
    In the case of \Cref{fig:fourpointsquare}, we can see that at $\delta = 1$ a cycle forms in $\mathcal{R}_1(d_X)$, which persists for all $\delta \in [1, \sqrt{2})$, until at $\delta = \sqrt{2}$ the entire square is filled in, which makes $\mathcal{R}_{\sqrt{2}}(d_X)$ contractible. As such, the Rips filtration of $X$ has nontrivial degree-1 homology for $\delta \in [1, \sqrt{2})$.

    On the other hand, in the case of \Cref{fig:fourpointrhombus}, in the Rips complex $\mathcal{R}_\delta(d_Y)$ the minor diagonal of the rhombus is filled in at $\delta = 1$, which is the same scale at which the cycle of outer edges forms: as such, $H_1(\mathcal{R}_\delta(d_Y))$ is trivial for all $\delta \geq 0$. 

    These spaces are part of a general pattern: in the case of $4$-point metric spaces and the degree-1 homology, there can be at most one point in the persistence diagram of the space. If it exists, the birth time is given by the maximum length of all edges in the quadrilateral, and the death time is given by the length of the minor diagonal.

    This pattern generalizes to higher dimensions: when considering the $k$th Rips persistent homology of a metric space with $(2k+2)$ points, there can be at most $1$ point in the persistence diagram, which can be calculated in closed form by similar formulae, as derived by G\'omez and M\'emoli in \cite{curvaturesets}.

    We now generalize this to the dynamic setting. First, we make the following definition:
    \begin{definition}[$n$-DMS]
        For any $n \in \mathbb{N}$, an $n$-DMS is a dynamic metric space $\gamma_X = (X, d_X(\cdot))$ with $|X| = n$: that is, exactly $n$ points.
    \end{definition}
    We will be interested in studying $H_k(\mathcal{R}^{\mathrm{lev}}(\gamma_X))$ of a $(2k+2)$-DMS $\gamma_X$, for any $k \in \mathbb{Z}_{\geq 0}$.
    \begin{remark} \label{rem:moduledimension}
        If $\gamma_X$ is a $(2k+2)$-DMS, the module $H_k(\mathcal{R}^{\mathrm{lev}}(\gamma_X))$ must be thin. This follows immediately from \cite[Theorem 4.4(B)]{curvaturesets} because for each $p \in \bdyn$, $\mathcal{R}^{\mathrm{lev}}(\gamma_X)(p)$ is a semimetric space with exactly $(2k+2)$ points. An abridged statement of \cite[Theorem 4.4(B)]{curvaturesets} can be found in \Cref{app:curvaturesetstheorem}.

        In particular, the fact that $H_k(\mathcal{R}^{\mathrm{lev}}(\gamma_X))$ is thin implies that its support is well-defined.
    \end{remark}
    To prove \Cref{thm:intervaldecomp}, one of our main structural results about the dynamic curvature set construction, we will rely on the following lemma:
    \begin{lemma}\label{lem:structuremaps}
        Suppose $k \in \mathbb{Z}_{\geq 0}$ and let $\gamma_X = (X, d_X(\cdot))$ be a $(2k+2)$-DMS. Put $\bbv = H_k(\mathcal{R}^{\mathrm{lev}}(\gamma_X))$. Then for any $p, q \in \supp \bbv$ satisfying $p \leq q$, we have $\mathcal{R}^{\mathrm{lev}}(\gamma_X)(p) = \mathcal{R}^{\mathrm{lev}}(\gamma_X)(q)$ as abstract simplicial complexes. Moreover, the structure map $\bbv(p \leq q)$ is an isomorphism.
    \end{lemma}
    \begin{proof}
        By definition of $\bdyn$, there exist intervals $I_1, I_2$ and $\delta_1, \delta_2 \in \mathbb{R}_+$ such that $p = (I_1, \delta_1)$ and $q = (I_2, \delta_2)$. Because $p, q \in \supp \bbv$,
        \[H_k(\mathcal{R}_{\delta_1}(\gamma_X(I_1)) \cong H_k(\mathcal{R}_{\delta_2}(\gamma_X(I_2)) \cong \mathbb{F},\]
        which means that $\mathrm{dgm}_k^{\mathrm{VR}}(\gamma_X(I_1))$ and $\mathrm{dgm}_k^{\mathrm{VR}}(\gamma_X(I_2))$ are both non-empty. By \cite[Theorem 4.3]{curvaturesets} (re-stated in \Cref{app:curvaturesetstheorem}), this implies that
        \[\mathcal{R}^{\mathrm{lev}}(\gamma_X)(p) = \mathcal{R}_{\delta_1}(\gamma_X(I_1)) \cong \mathfrak{B}_{k+1} \quad \mathrm{ and } \quad \mathcal{R}^{\mathrm{lev}}(\gamma_X)(q) = \mathcal{R}_{\delta_2}(\gamma_X(I_2)) \cong \mathfrak{B}_{k+1} \]
        as abstract simplicial complexes, where $\mathfrak{B}_{k+1}$ is the $(k+1)$-dimensional cross polytope. And so \[\mathcal{R}^{\mathrm{lev}}(\gamma_X)(p) \cong \mathcal{R}^{\mathrm{lev}}(\gamma_X)(q).\] But by functoriality, because $p \leq q$ we know that $\mathcal{R}^{\mathrm{lev}}(\gamma_X)(p) \subseteq \mathcal{R}^{\mathrm{lev}}(\gamma_X)(q)$: by combining this with the fact that $\mathcal{R}^{\mathrm{lev}}(\gamma_X)(p) \cong \mathcal{R}^{\mathrm{lev}}(\gamma_X)(q)$, we get we get $\mathcal{R}^{\mathrm{lev}}(\gamma_X)(p) = \mathcal{R}^{\mathrm{lev}}(\gamma_X)(q)$. And the structure map $\bbv(p \leq q)$ is induced by the inclusion map $\iota: \mathcal{R}^{\mathrm{lev}}(\gamma_X)(p) \rightarrow \mathcal{R}^{\mathrm{lev}}(\gamma_X)(q)$, which is a simplicial isomorphism. So $\bbv(p \leq q)$ is an isomorphism.
    \end{proof}
    Using the above description, we can obtain the following:
    \begin{remark}[Convexity of supports]\label{lem:suppconvexity}
        Suppose $\gamma_X$ is a $(2k+2)$-DMS. Then $\supp H_k(\mathcal{R}^{\mathrm{lev}}(\gamma_X))$ is a convex subset of $\bdyn$.
    \end{remark}
    \begin{proof}
        Put $\bbv = H_k(\mathcal{R}^{\mathrm{lev}}(\gamma_X))$, and let $p, q, r \in \bdyn$ such that $p \leq q \leq r$ and satisfying $p, r \in \supp \bbv$. To show convexity, we must show that $q \in \supp \bbv$. From earlier assumptions, the structure map $\bbv(p \leq r)$ must be an isomorphism, i.e. has rank $1$. But by functoriality, $\bbv(p \leq r) = \bbv(q \leq r) \circ \bbv(p \leq q)$, so the maps $\bbv(q \leq r): \bbv(q) \rightarrow \bbv(r)$ and $\bbv(p \leq q): \bbv(p) \rightarrow \bbv(q)$ must both have rank $1$. Therefore, $\bbv(q)$ cannot be trivial, and thus $q \in \supp \bbv$.
    \end{proof}
    This results in a structure theorem for $H_k(\mathcal{R}^{\mathrm{lev}}(\gamma_X))$ of a $(2k+2)$-DMS $\gamma_X$. In particular, we show both interval decomposability and a constructive description of the interval decomposition.
    \begin{theorem}[Interval Decomposability] \label{thm:intervaldecomp}
        Suppose $k \in \mathbb{Z}_{\geq 0}$, and $\gamma_X = (X, d_X(\cdot))$ is a $(2k+2)$-DMS. Then $H_k(\mathcal{R}^{\mathrm{lev}}(\gamma_X))$ is interval-decomposable. Moreover, $H_k(\mathcal{R}^{\mathrm{lev}}(\gamma_X)) \cong \bigoplus_{\alpha \in A} \mathbb{I}_{S_\alpha}$, where $\{S_\alpha\}_{\alpha \in A}$ are the order-connected components of $\supp H_k(\mathcal{R}^{\mathrm{lev}}(\gamma_X))$.
    \end{theorem}
    \begin{proof}
        Let $\bbv = H_k(\mathcal{R}^{\mathrm{lev}}(\gamma_X))$, and let $\{S_\alpha\}_{\alpha \in A}$ be the decomposition of $\supp \bbv$ into order-connected components. We will first show that $\bbv \cong \mathbb{I}_{(\supp \bbv)}$.

        By \Cref{lem:structuremaps}, for any $p, q \in \supp \bbv$ such that $p \leq q$, it follows that $\mathcal{R}^{\mathrm{lev}}(\gamma_X)(p) = \mathcal{R}^{\mathrm{lev}}(\gamma_X)(q)$ as simplicial complexes. Thus, for any fixed $\alpha \in A$, $\mathcal{R}^{\mathrm{lev}}(\gamma_X)(p) = \mathcal{R}^{\mathrm{lev}}(\gamma_X)(q)$ for all $p, q \in S_\alpha$. As the underlying simplicial complexes are equal, we can identify their chain complexes. For each $\alpha$, choose an arbitrary $p_\alpha \in S_\alpha$. Because $\bbv(p_\alpha) = H_k(\mathcal{R}^{\mathrm{lev}}(\gamma_X)(p_\alpha)) \cong \mathbb{F}$, we can pick a representative cycle $[\xi_\alpha]$ that generates $\bbv(p_\alpha)$. And thus $[\xi_\alpha]$ generates $\bbv(q)$, for all $q \in S_\alpha$.

        For any $p \in \bdyn$, we can define the maps $f_p: \bbv(p) \rightarrow \mathbb{I}_{(\supp \bbv)}(p)$ and $g_p: \mathbb{I}_{(\supp \bbv)}(p) \rightarrow \bbv(p)$ by
        \[
            f_p(x) = 
            \begin{cases}
                t & \exists!\, \alpha \in A \text{ s.t. } p \in S_\alpha, x = t [\xi_\alpha] \\
                0 & \text{otherwise}
            \end{cases}
            \qquad
            g_p(t) = 
            \begin{cases}
                t[\xi_\alpha] & \exists! \, \alpha \in A \text{ s.t. } p \in S_\alpha \\
                0 & \text{otherwise}
            \end{cases}
        \]

        Which satisfy naturality and are inverses of each other. Thus, $\bbv \cong \mathbb{I}_{(\supp \bbv)}$. From this, it follows that $\bbv \cong \mathbb{I}_{(\supp \bbv)} \cong \bigoplus_{\alpha \in A} \mathbb{I}_{S_\alpha}$, which is the interval decomposition of $\bbv$.
    \end{proof}
    
    The following is immediate by \Cref{thm:intervaldecomp} and \Cref{lem:suppconvexity}:
    \begin{corollary} \label{cor:curvaturesetsacd}
        For any $k \in \mathbb{Z}_{\geq 0}$ and $\gamma_X$ a $(2k+2)$-DMS, the module $H_k(\mathcal{R}^{\mathrm{lev}}(\gamma_X))$ is ACD, in the sense of \Cref{def:acdmodules}.
    \end{corollary}
    \begin{proof}
        It is known that $H_k(\mathcal{R}^{\mathrm{lev}}(\gamma_X)) \cong \bigoplus_{\alpha \in A} \mathbb{I}_{S_\alpha}$ for a family of intervals $\{S_\alpha\}_{\alpha \in A}$. All that is left is to verify incomparability. Suppose this doesn't hold, and so for some $\alpha \neq \beta$ we have $p \in S_\alpha, q \in S_\beta$ such that $p \leq q$. But this cannot be, because the order-connected components $\{S_\alpha\}_{\alpha \in A}$ are maximal, and the existence of such $p, q$ imply that $S_\alpha \cup S_\beta$ is order-connected, a contradiction.
    \end{proof}
    In general, interval-decomposable modules are atypical in the space of multiparameter persistence modules: in particular, Bauer and Scoccola have shown in \cite{multiparameterindecomposable} that indecomposable persistence modules are dense in the space of multiparameter persistence modules.

    The fact that our construction produces ACD modules provides substantial computational benefits, which will be leveraged in the coming sections.

\section{Generalization to Arbitrary Size Data}
    To leverage the structure theorems \Cref{thm:intervaldecomp} and \Cref{cor:curvaturesetsacd} for analyzing dynamic data, we must contend with the fact that real-world data sets can be of arbitrary cardinality, certainly much greater than $(2k+2)$ for any reasonable $k$.

    To generalize our approach, we will consider all $(2k+2)$ size subsets of the data set. In practice, we will take a sample of these subsets, and a $k$-center scheme is described in \Cref{sect:experimental} to choose a manageable number of subsets while maintaining as much representativeness as possible.
    \begin{definition}[Curvature set of a DMS]
        Let $\gamma_X = (X, d_X)$ be a dynamic metric space. The \emph{$n$th curvature set}\footnote{The name ``curvature set" was coined by Gromov in \cite{Gromov1999MetricSF} due to the fact these subsets can be used to recover curvature information from manifolds.} $K_n(\gamma_X)$ is the set of dynamic metric spaces
        \[K_n(\gamma_X) \defeq \{ (S, d_X |_S) : S \subseteq X, |S| \leq n \}.\]
    \end{definition}

    We can then apply the Rips spatiotemporal persistent homology to each element of the curvature set. 
    \begin{definition}[Persistence set of a DMS]\label{def:dyncurvpmod}
        Let $\gamma_X = (X, d_X)$ be a dynamic metric space, and $k \in \mathbb{Z}_{\geq 0}$. The \emph{$(n, k)$th persistence set}, denoted $D_k^n(\gamma_X)$ is
        \[D_k^n(\gamma_X) \defeq \{ H_k(\mathcal{R}^{\mathrm{lev}}(\gamma_S)) : \gamma_S \in K_{n}(\gamma_X) \}.\]
    \end{definition}

    For a dynamic metric space $\gamma_X$, $D_k^n(\gamma_X)$ is the invariant we will work with. Most often, we will consider $D_k^{(2k+2)}(\gamma_X)$, because the structure of persistence modules in $D_k^{(2k+2)}(\gamma_X)$ are particularly simple by earlier results. To reduce notational burden, define $D_k(\gamma_X) := D_k^{(2k+2)}(\gamma_X)$.

    To metrize the space of all persistence sets, we have the following metric:
    \begin{definition}[Hausdorff distance on persistence sets] \label{def:persistencesetmetric}
        Let $n, k \in \mathbb{Z}_{\geq 0}$, and let $\gamma_X$ and $\gamma_Y$ be two DMSs. For any metric $\rho$ between $\bdyn$ indexed modules, we have the associated \emph{Hausdorff distance} between persistence sets induced by $\rho$:
        \[d^{\rho}_H(D_k^n(\gamma_X), D_k^n(\gamma_Y)) = \max \left\{ \sup_{\bbv \in D_k^n(\gamma_X)} \inf_{\bbw \in D_k^n(\gamma_Y)} \rho(\bbv, \bbw), \sup_{\bbw \in D_k^n(\gamma_Y)} \inf_{\bbv \in D_k^n(\gamma(X))} \rho(\bbv, \bbw) \right\}.\]
        For notational simplicity, we will use $d_H(D_k^n(\gamma_X), D_k^n(\gamma_Y))$ to denote $d^{d_I^{\bvec}}_H(D_k^n(\gamma_X), D_k^n(\gamma_Y))$. That is, the Hausdorff distance induced by $d_I^{\bvec}$ (\Cref{rem:dyninterleavingdistance}).
    \end{definition}

    This construction of a persistence set from a given dynamic metric space is stable:
    \begin{restatable}[Stability of dynamic curvature sets]{theorem}{stabrestate}\label{thm:stability}
        Let $\gamma_X = (X, d_X)$ and $\gamma_Y = (Y, d_Y)$ be two dynamic metric spaces. Then for any $k \in \mathbb{Z}_{\geq 0}$ and $n \in \mathbb{N}$, \[d_H\left(D_k^n(\gamma_X), D_k^n(\gamma_Y)\right) \leq 2 \cdot d_{\text{dyn}}(\gamma_X, \gamma_Y).\]
        where $d_{\text{dyn}}$ is the $\lambda=2$ generalized Gromov-Hausdorff distance from \cite{stablesignatures}; see \Cref{def:ddynmetric}.
    \end{restatable}

    In the next section, we will explore an additional metric $d_E$ on $\bdyn$ indexed modules that induces a Hausdorff distance $d_H^{d_E}$ which is also stable; see \Cref{cor:dEhausdorffstable}.
    
    For a self-contained exposition on $d_{\text{dyn}}$, see \Cref{app:lambdaslackdistance}. For a proof of \Cref{thm:stability}, refer to \Cref{sec:stabilityproof}.

\section{The Erosion Distance}
    In this section we will recall the erosion distance (first introduced by Patel in \cite{patelerosion} and later generalized by Puuska in \cite{puuskaerosion}), and explain its connections to other metrics, such as the interleaving distance $d_I^{\bvec}$. This will be the main target for our algorithmic work.

    For convenience within this section, when the poset $\mathcal{P}$ is not specified, we should interpret $\mathbf{Seg}$ to mean $\mathbf{Seg}(\mathbb{R}^d)$ where $\mathbb{R}^d$ is regarded as a poset category with the product order. See \Cref{def:segment}.
    
    \subsection{The Erosion Distance $d_E$ on $\mathbb{R}^n$}
        Consider the $d$ dimensional poset $\mathbb{R}^d$ with the product order. 

        \begin{definition}
            The \emph{$\varepsilon$-thickening} (for nonnegative $\varepsilon\in\mathbb{R}$) of a segment $I=[p,\,q]\in\mathbf{Seg}$ is the segment $[p^{-\varepsilon},\,q^{+\varepsilon}]$ denoted by $I^{+\varepsilon}$. If $p\in\mathbb{R}^d$ then $p^{+\varepsilon}=\sum_i (p^{(i)} +\varepsilon )\vec{e}_i$ and $p^{-\varepsilon}=\sum_i (p^{(i)} -\varepsilon )\vec{e}_i$.
        \end{definition}

        Note that $I\leq I^{+\varepsilon}$ for all $I\in\mathbf{Seg}$. Now we properly define the rank invariant of a $\mathbb{R}^d$ module.

        \begin{definition}[\cite{rankinv}]
            Given a $\mathbb{R}^d$ module $\mathbb{V}:\mathbb{R}^d \to \bvec$, its \emph{rank invariant} is a functor $\mathrm{rk}_\mathbb{V}:\mathbf{Seg}\to\mathbb{N}^\text{op}$ which sends $[p,\,q]\in\mathbf{Seg}$ to the rank of the structure map from $p$ to $q$ in $\mathbb{V}$.
        \end{definition}

        Note that we are asserting that for any $\mathbb{R}^d$ module $\mathbb{V}$ and $I\leq J\in\mathbf{Seg}$ (and in particular for $I\leq I^{+\varepsilon}$) we have $\mathrm{rk}_\mathbb{V}(J)\leq\mathrm{rk}_\mathbb{V}(I)\in\mathbb{N}$. 

        Now we specialize the erosion distance \cite{patelerosion} to obtain a metric between rank invariants of $\mathbb{R}^d$ module.

        \begin{definition}
            The \emph{erosion distance} $d_\text{E}$ between rank invariants  $\mathrm{rk}_\mathbb{V}$ and $\mathrm{rk}_\mathbb{W}$ is defined by
            \[
            d_\text{E}(\mathrm{rk}_\mathbb{V},\,\mathrm{rk}_\mathbb{W}) \defeq \inf\Bigl\{ \varepsilon>0\;\Bigr\rvert \; \forall I\in\operatorname{\mathbf{Seg}},\, \mathrm{rk}_\mathbb{V}(I^{+\varepsilon})\leq \mathrm{rk}_\mathbb{W}(I) ,\, \mathrm{rk}_\mathbb{W}(I^{+\varepsilon})\leq \mathrm{rk}_\mathbb{V}(I) \Bigr\}.
            \]
        \end{definition}
    \subsection{The Erosion Distance $d_E$ on $\bdyn$}
        We now introduce the Erosion Distance $d_E$ on $\bdyn$-indexed modules, with presentation in the style of \cite{spatiotemporal}. This is similar to the $\mathbb{R}^n$, as $\bdyn$ can be identified with a subset of \emph{admissible points} in $\mathbb{R}^{\mathrm{op}} \times \mathbb{R} \times \mathbb{R}$.

        \begin{definition}[\cite{spatiotemporal}]
            Define the poset $\mathbb{R}_{\times}^3$ to be $\mathbb{R}^{\mathrm{op}} \times \mathbb{R} \times \mathbb{R}$ with the product order. Define the poset $\mathbb{R}_{\times}^6$ to be $\mathbb{R} \times \mathbb{R}^{\mathrm{op}} \times \mathbb{R}^{\mathrm{op}} \times \mathbb{R}^{\mathrm{op}} \times \mathbb{R} \times \mathbb{R}$ with the product order.

            Intuitively, $\mathbb{R}_{\times}^6$ should be viewed as $\left(\mathbb{R}_{\times}^3\right)^{\mathrm{op}} \times \mathbb{R}_{\times}^3$.
        \end{definition}

        \begin{definition}[Admissible point \cite{spatiotemporal}]
            We say a point $(t_1, t_2, \delta) \in \mathbb{R}_{\times}^3$ is \emph{admissible} if $t_1 \leq t_2$ and $\delta \geq 0$. We can identify the poset $\bdyn$ with the sub-poset of admissible points in $\mathbb{R}_{\times}^3$.

            We say a point $\mathbf{a} = (a_1, a_2, a_3, a_4, a_5, a_6) \in \mathbb{R}_{\times}^6$ is \emph{admissible} if $(a_1, a_2, a_3), (a_4, a_5, a_6) \in \bdyn$ (under the identification above) and that $(a_1, a_2, a_3) \leq (a_4, a_5, a_6)$ in $\mathbb{R}_{\times}^3$. Note that if such an $\mathbf{a}$ is admissible, we can identify $\mathbf{a}$ with the segment $[(a_1, a_2, a_3), (b_1, b_2, b_3)]$ in $\mathbf{Seg}({\bdyn})$, after identifying $(a_1, a_2, a_3)$ and $(a_4, a_5, a_6)$ with points in $\bdyn$.
            
            If a point $\mathbf{a} \in \mathbb{R}_{\times}^6$ is not admissible, we say $\mathbf{a}$ is \emph{non-admissible}. Furthermore, a non-admissible point $\mathbf{a} \in \mathbb{R}_{\times}^6$ is called \emph{trivially non-admissible} if there exists no admissible $\mathbf{b} \in \mathbb{R}_{\times}^6$ such that $\mathbf{b} < \mathbf{a}$.
        \end{definition}

        With this, we are now prepared to define the rank invariant on $\bdyn$.
        \begin{definition}[Adapted rank invariant on $\bdyn$ \cite{spatiotemporal}] \label{def:bdynrankinvariant}
            Suppose $\bbv: \bdyn \rightarrow \bvec$ is a persistence module. The \emph{adapted rank invariant} of $\bdyn$ is the functor $\mathrm{rk}_\bbv: \mathbb{R}_{\times}^6 \rightarrow \mathbb{Z}_+$ defined as follows: for $\mathbf{a} = (a_1, \ldots, a_6)$,
            \[\mathrm{rk}_\bbv(\mathbf{a}) := \begin{cases} \mathrm{rank}(\bbv(([a_1, a_2], a_3) \leq ([a_4, a_5], a_6)) & \mathbf{a} \text{ is admissible} \\ \infty & \mathbf{a} \text{ is trivially non-admissible} \\ 0 & \text{otherwise} \end{cases}\]
            The fact that $\mathrm{rk}_\bbv$ is a functor can be seen in \cite[Proposition 4.3]{spatiotemporal}.
        \end{definition}

        To define a notion of Erosion distance, we need a notion of thickening:
        \begin{definition}
            For a point $\mathbf{p} \in \mathbb{R}_{\times}^3$ and an $\varepsilon > 0$, define the \emph{shift} $\mathbf{p}^{+\varepsilon} := \mathbf{p} + \varepsilon \cdot (-1, 1, 1)$. Analogously define the shift $\mathbf{p}^{-\varepsilon} := \mathbf{p} - \varepsilon \cdot (-1, 1, 1)$.

            For a point $\mathbf{a} \in \mathbb{R}_{\times}^6$ and an $\varepsilon > 0$, define the shift 
            \[ \mathbf{a}^{+\varepsilon} := \mathbf{a} + \varepsilon \cdot (1, -1, -1, -1, 1, 1) \]
            and analogously 
            \[ \mathbf{a}^{-\varepsilon} := \mathbf{a} - \varepsilon \cdot (1, -1, -1, -1, 1, 1). \]
        \end{definition}
        Note that while the notation for $+\epsilon$ over $\mathbb{R}_{\times}^6$ is the same as $+\epsilon$ over $\mathbf{Seg}$, they are not immediately the same operation. When referring to variables in $\mathbb{R}_{\times}^6$, we will use boldface to distinguish them.
        
        The functoriality of the adapted rank invariant allows us to construct the erosion distance on $\bdyn$:
        \begin{definition}
            For $\bbv, \bbw: \bdyn \rightarrow \bvec$, the \emph{erosion distance} $d_\text{E}$ between adapted rank invariants $\mathrm{rk}_\mathbb{V}$ and $\mathrm{rk}_\mathbb{W}$ is defined by
            \[
            d_\text{E}(\mathrm{rk}_\mathbb{V},\,\mathrm{rk}_\mathbb{W}) \defeq \inf\Bigl\{ \varepsilon>0\;\Bigr\rvert \; \forall I\in \mathbb{R}_{\times}^6,\, \mathrm{rk}_\mathbb{V}(I^{+\varepsilon})\leq \mathrm{rk}_\mathbb{W}(I) ,\, \mathrm{rk}_\mathbb{W}(I^{+\varepsilon})\leq \mathrm{rk}_\mathbb{V}(I) \Bigr\}.
            \]
        \end{definition}
        Note that this is similar to the earlier erosion distance on $\mathbb{R}^d$, except to bypass issues of admissibility we work over $\mathbb{R}_{\times}^6$ instead of $\mathbf{Seg}(\bdyn)$.
    
    \subsection{Application of the erosion distance}
        \begin{remark}\label{rem:erosionlowerboundsinterleaving}
            By \cite[Theorem 8.2]{patelerosion} the erosion distance is a lower bound for the interleaving distance. More precisely in the case of $\bdyn$ indexed modules $\bbv, \bbw$ the following holds: \[d_E(\mathrm{rk}_\bbv, \mathrm{rk}_\bbw) \leq d_I^{\bvec}(\bbv, \bbw).\]
        \end{remark}
        This motivates the erosion distance as a computational proxy for the interleaving distance. As the erosion distance is a metric on $\bdyn$ indexed modules, we have its associated Hausdorff distance on persistence sets:
        \begin{definition}\label{def:persistenceseterosion}
            For $n, k \in \mathbb{Z}_+$ and $\gamma_X, \gamma_Y$ two dynamic metric spaces, we make the definition \[d_H^{\mathrm{ero}}(D^n_k(\gamma_X), D^n_k(\gamma_Y)) := d_H^{d_E}(D^n_k(\gamma_X), D^n_k(\gamma_Y)),\] where $d_E$ is the erosion distance between rank invariants of $\bdyn$ modules. Note that this is a specialization of \Cref{def:persistencesetmetric}.
        \end{definition}
        Which leads to an immediate corollary,
        \begin{corollary} \label{cor:dEhausdorffstable}
            Combining \Cref{rem:erosionlowerboundsinterleaving}, \Cref{def:persistenceseterosion} and \Cref{thm:stability}, we can see that for any $n, k \in \mathbb{Z}_{\geq 0}$ and $\gamma_X, \gamma_Y$ two DMSs, we have \[d_H^{\mathrm{ero}}(D^n_k(\gamma_X), D^n_k(\gamma_Y)) \leq d_H(D^n_k(\gamma_X), D^n_k(\gamma_Y)) \leq 2 \cdot d_{\mathrm{dyn}}(\gamma_X, \gamma_Y).\]
        \end{corollary}
        The computation of $d_E$ and $d_E^{\mathrm{ero}}$ will be addressed in the remainder of the algorithmic and experimental sections. We will address the case of computing erosion distance between rank invariants of $\mathbb{R}^d$-indexed modules.

\section{Algorithmic Results For Antichain-Decomposable Modules}

    In this section we introduce a new algorithm for computing a variant of the erosion distance introduced in \cite{patelerosion}, between any two $\mathbb{R}^d$-indexed antichain-decomposable modules (see \Cref{def:acdmodules}) after discretization. In particular, if the discretized domain is $[n]^d$, then the runtime of our algorithm is $O(n^{d-1} d \log{n})$. It is worth emphasizing that this holds for arbitrary poset-indexed ACD modules, and is independent from the $(2k+2)$ construction.

    \begin{framed}
    \textit{All persistence modules in this section should be assumed to be antichain-decomposable (see \Cref{def:acdmodules}). In particular, recall that ACD modules are thin and have convex support (see \Cref{rem:acdproperties}). Additionally, all $\mathbb{R}^d$-modules and $\bdyn$-modules are assumed to be bounded in the sense of \Cref{def:boundedacdmodule}.}
    \end{framed}    
    \begin{remark}
        While the results in this section are stated for ACD modules, they hold in slightly more generality. In particular, as the erosion distance is a metric on the rank invariants, only the rank of structure maps matters. As such, the algorithm outlined in \Cref{alg:dEalgo} is applicable to any thin module with convex support. This is generalized further in \Cref{sect:dealgoextension}, where the thinness requirement is relaxed.
    \end{remark}
        \subsection{Simplifications of $d_E$}

            To understand the algorithm in \Cref{dEquadratic} we will develop some results about the rank invariant in the thin module case. First we reformulate the definition of the erosion distance to emphasize that each segment in $\mathbf{Seg}$ has some threshold $\varepsilon$ value above which it satisfies the rank invariant inequalities in the erosion distance definition.

            \begin{definition}
                For $I\in\mathbf{Seg}$ and ACD $\mathbb{R}^d$ modules $\mathbb{V}$ and $\mathbb{W}$, we define \[\mathcal{E}_{\mathbb{V},\mathbb{W}}(I)=\inf\{ \varepsilon>0 \; |\; \mathrm{rk}_\mathbb{V}(I^{+\varepsilon})\leq \mathrm{rk}_\mathbb{W}(I) ,\, \mathrm{rk}_\mathbb{W}(I^{+\varepsilon})\leq \mathrm{rk}_\mathbb{V}(I)\}.\]
            \end{definition}
            Which allows for a simpler formulation of $d_E$, namely:
            \begin{remark}\label{rem:dEsup}
                We see that
                \[d_E(\mathrm{rk}_\mathbb{V},\,\mathrm{rk}_\mathbb{W})=\sup_{I\in\mathbf{Seg}} \mathcal{E}_{\mathbb{V},\mathbb{W}}(I).
                \]
            \end{remark}
            \begin{definition}
                 For $I\in\mathbf{Seg}$ and an ACD $\mathbb{R}^d$ module $\mathbb{V}$, we define \[\mathcal{E}_\mathbb{V}(I):=\inf\{\varepsilon>0\;|\;\mathrm{rk}_\mathbb{V}(I^{+\varepsilon})=0\}.\]
            \end{definition}
            Now, we will analyze the behavior of $\mathcal{E}_{\bbv, \bbw}$.
            \begin{lemma} \label{dEepscasework}
                Let $\mathbb{V}$, $\mathbb{W}$ be ACD $\mathbb{R}^d$ modules and $I\in\mathbf{Seg}$.

                \begin{itemize}
                    \item If $\mathrm{rk}_\mathbb{V}(I)=0$ and $\mathrm{rk}_\mathbb{W}(I)=0$ then $\mathcal{E}_{\mathbb{V},\,\mathbb{W}}(I)=0$.
                    \item If $\mathrm{rk}_\mathbb{V}(I)=1$ and $\mathrm{rk}_\mathbb{W}(I)=0$ then $\mathcal{E}_{\mathbb{V},\,\mathbb{W}}(I)=\mathcal{E}_\mathbb{V}(I)$.
                    \item If $\mathrm{rk}_\mathbb{V}(I)=0$ and $\mathrm{rk}_\mathbb{W}(I)=1$ then $\mathcal{E}_{\mathbb{V},\,\mathbb{W}}(I)=\mathcal{E}_\mathbb{W}(I)$.
                    \item If $\mathrm{rk}_\mathbb{V}(I)=1$ and $\mathrm{rk}_\mathbb{W}(I)=1$ then $\mathcal{E}_{\mathbb{V},\,\mathbb{W}}(I)=0$.
                \end{itemize}
            \end{lemma}
            \begin{proof}
                In the cases that $\mathrm{rk}_\mathbb{V}(I)=\mathrm{rk}_\mathbb{W}(I)$ the condition that $\mathrm{rk}_\mathbb{V}(I^{+\varepsilon})\leq \mathrm{rk}_\mathbb{W}(I) ,\, \mathrm{rk}_\mathbb{W}(I^{+\varepsilon})\leq \mathrm{rk}_\mathbb{V}(I)$ is satisfied for $\varepsilon=0$. 

                In the case that $\mathrm{rk}_\mathbb{V}(I)=1$ and $\mathrm{rk}_\mathbb{W}(I)=0$, the condition that $\mathrm{rk}_\mathbb{W}(I^{+\varepsilon})\leq \mathrm{rk}_\mathbb{V}(I)$ is satisfied for $\varepsilon=0$. Thus $\mathcal{E}_{\mathbb{V},\,\mathbb{W}}(I)$ will evaluate to the minimal $\varepsilon$ with $\mathrm{rk}_\mathbb{V}(I^{+\varepsilon})\leq \mathrm{rk}_\mathbb{W}(I)=0$ or equivalently with $\mathrm{rk}_\mathbb{V}(I^{+\varepsilon})=0$. Which is precisely $\mathcal{E}_\mathbb{V}(I)$.

                The final case symmetrically follows from the previous one.
            \end{proof}
            \begin{definition}
                For ACD $\mathbb{R}^d$ modules $\mathbb{V}$ and $\mathbb{W}$ we define \[\mathbf{Seg}_{\mathbb{V},\,\mathbb{W}}:=\{I\in\mathbf{Seg}\;|\;\mathrm{rk}_\mathbb{V}(I)=1,\,\mathrm{rk}_\mathbb{W}(I)=0\}.\]
            \end{definition}
            \begin{remark} \label{dEsymchar}
                For $I\in\mathbf{Seg}$ and ACD $\mathbb{R}^d$ modules 
                $\mathbb{V}$ and $\mathbb{W}$, $I\notin \mathbf{Seg}_{\mathbb{V},\,\mathbb{W}}\cup\mathbf{Seg}_{\mathbb{W},\,\mathbb{V}}$ implies $\mathcal{E}_{\mathbb{V},\,\mathbb{W}}(I)=0$ (by \Cref{dEepscasework}), thus by \Cref{rem:dEsup},
                \[d_E(\mathrm{rk}_\mathbb{V},\,\mathrm{rk}_\mathbb{W})=\sup_{I\in\mathbf{Seg}} \mathcal{E}_{\mathbb{V},\mathbb{W}}(I)=\sup_{I\in\mathbf{Seg}_{\mathbb{V},\,\mathbb{W}}\cup\mathbf{Seg}_{\mathbb{W},\,\mathbb{V}}} \mathcal{E}_{\mathbb{V},\mathbb{W}}(I).\]
                Hence
                \[d_E(\mathrm{rk}_\mathbb{V},\,\mathrm{rk}_\mathbb{W})=\max\left(\sup_{I\in\mathbf{Seg}_{\mathbb{V},\,\mathbb{W}}}\mathcal{E}_{\mathbb{V},\mathbb{W}}(I),\,\sup_{I\in\mathbf{Seg}_{\mathbb{W},\,\mathbb{V}}}\mathcal{E}_{\mathbb{V},\mathbb{W}}(I)\right).
                \]
                By applying \Cref{dEepscasework} once more we see that \[d_E(\mathrm{rk}_\mathbb{V},\,\mathrm{rk}_\mathbb{W})=\max\left(\sup_{I\in\mathbf{Seg}_{\mathbb{V},\,\mathbb{W}}}\mathcal{E}_{\mathbb{V}}(I),\,\sup_{I\in\mathbf{Seg}_{\mathbb{W},\,\mathbb{V}}}\mathcal{E}_{\mathbb{W}}(I)\right).\]
            \end{remark}

            We now develop some additional reformulations to eventually minimize the dimension of the subset of elements of $\mathbb{R}^d$ we must include in the suprema above, this will eventually lead to reducing computational complexity.
            \begin{definition} \label{dEshiftdef}
                For ACD $\mathbb{R}^d$ modules $\mathbb{V}$, $\mathbb{W}$, and $p,\,q\in\mathbb{R}^d$ we define:
                \begin{itemize}
                    \item $\mathbb{R}^d_{\mathbb{V},\,\mathbb{W}}:=\{u\in\mathbb{R}^d\;|\;\mathbb{V}(u)\cong\mathbb{F},\,\mathbb{W}(u)\cong0\},$

                    \item $\mathbb{R}^d_{\mathbb{V}}:=\{u\in\mathbb{R}^d\;|\;\mathbb{V}(u)\cong\mathbb{F}\},$

                    \item $\mathcal{E}^\text{small}_\mathbb{V}(p):=\sup_{q\geq p\in{\mathbb{R}^d}}\mathcal{E}_\mathbb{V}([p,\,q]),$

                    \item $\mathcal{E}^\text{large}_\mathbb{V}(q):=\sup_{p\leq q\in{\mathbb{R}^d}}\mathcal{E}_\mathbb{V}([p,\,q]).$
                \end{itemize}
            \end{definition}
            \begin{remark}\label{dEcubic}
                By convexity we see that for $[p,\,q]\in\mathbf{Seg}$ and ACD $\mathbb{R}^d$ modules $\mathbb{V}$ and $\mathbb{W}$, $[p,\,q]\in\mathbf{Seg}_{\mathbb{V},\,\mathbb{W}}$ if and only if $p\in\mathbb{R}^d_{\mathbb{V},\,\mathbb{W}}$ and $q\in\mathbb{R}^d_\mathbb{V}$, or $q\in\mathbb{R}^d_{\mathbb{V},\,\mathbb{W}}$ and $p\in\mathbb{R}^d_\mathbb{V}$. Thus \[\sup_{I\in\mathbf{Seg}_{\mathbb{V},\,\mathbb{W}}}\mathcal{E}_{\mathbb{V}}(I)=\max\left(\sup_{p\in\mathbb{R}^d_{\mathbb{V},\,\mathbb{W}}}\mathcal{E}^\text{small}_\mathbb{V}(p),\,\sup_{q\in\mathbb{R}^d_{\mathbb{V},\,\mathbb{W}}}\mathcal{E}^\text{large}_\mathbb{V}(q)\right).\] Applying \Cref{dEsymchar} we see that if we define
                \[
                    \mathcal{E}^\ast(\mathbb{V},\,\mathbb{W})=\max\left(\sup_{p\in\mathbb{R}^d_{\mathbb{V},\,\mathbb{W}}}\mathcal{E}^\text{small}_\mathbb{V}(p),\,\sup_{q\in\mathbb{R}^d_{\mathbb{V},\,\mathbb{W}}}\mathcal{E}^\text{large}_\mathbb{V}(q)\right)
                \]
                then 
                \[    d_E(\mathrm{rk}_\mathbb{V},\,\mathrm{rk}_\mathbb{W})=\max\left(\mathcal{E}^\ast(\mathbb{V},\,\mathbb{W}),\,\mathcal{E}^\ast(\mathbb{W},\,\mathbb{V})\right).
                \]
            \end{remark}

            \begin{lemma}\label{dEprojection}
                For any segment $[p, q]$ of $\mathbb{R}^d$, we have $\mathcal{E}_\mathbb{V}([p,\,p]) \geq \mathcal{E}_\mathbb{V}([p,\,q])$. Similarly, for such a $[p, q]$ we also have $\mathcal{E}_\mathbb{V}([q,\,q]) \geq \mathcal{E}_\mathbb{V}([p,\,q])$.

                As an immediate consequence, we have for any $p, q \in \mathbb{R}^d$ that \[\mathcal{E}^{\mathrm{small}}_\bbv(p) = \mathcal{E}_\bbv([p, p]) \qquad \mathcal{E}^{\mathrm{large}}_\bbv(q) = \mathcal{E}_\bbv([q, q]).\]
            \end{lemma}
            \begin{proof}
                We will show that $\mathcal{E}_\mathbb{V}([p,\,p]) \geq \mathcal{E}_\mathbb{V}([p,\,q])$, and the proof of $\mathcal{E}_\mathbb{V}([q,\,q]) \geq \mathcal{E}_\mathbb{V}([p,\,q])$ proceeds analogously.
                
                Let $\varepsilon > 0$ be arbitrary. Because $p \leq p^{+\varepsilon} \leq q^{+\varepsilon}$, we have that $[p, p]^{+\varepsilon} \subseteq [p, q]^{+\varepsilon}$ as segments. Because $\mathrm{rk}_\mathbb{V}$ is order-reversing, this implies that $\mathrm{rk}_\mathbb{V}([p, p]^{+\varepsilon}) \geq \mathrm{rk}_\mathbb{V}([p, q]^{+\varepsilon})$. Therefore, if $\mathcal{E}_{\bbv}([p, q]) \geq \varepsilon$ for any $\varepsilon$, then we also have that $\mathcal{E}_{\bbv}([p, p]) \geq \varepsilon$. Taking the infimum over $\varepsilon$ yields $\mathcal{E}_\mathbb{V}([p,\,p]) \geq \mathcal{E}_\mathbb{V}([p,\,q])$.
            \end{proof}

            By the results of this section, we have managed to manipulate $d_E$ into a form suitable for an algorithmic approach. 

        \subsection{Discretization and Algorithm}
            In order to represent the rank information of an antichain-decomposable $\bbv: \mathbb{R}^d \rightarrow \bvec$ module in a computer's memory, we introduce a notion of discretization. 
            Fix a discrete subset of $\mathbb{R}^d$: the lattice $[n]^d$ with the product order suffices (perhaps after translation or rotation). Then, we consider the $[n]^d$-indexed thin module whose support is $[n]^d \cap \supp \bbv$. This $[n]^d$ module will be the discretized version of $\mathbb{V}$ we will work with.
            
            Note that computing erosion distance between these discretized modules will always lower bound the erosion distance between undiscretized modules: this is because every segment in the discretized poset is also a segment in the undiscretized poset. 

            To quantify the precision of this approximation for $\bdyn$ we introduce parameters to the discretization process. Recalling that $\mathbf{Dyn}$ has two temporal and one spatial dimension, we choose the discrete sample as a sublattice with $S$ equally spaced points in the spatial direction and $T$ equally spaced points in each temporal direction. Hence the size of this subset is $ST^2$. However by convexity we can store it in $\mathcal{O}(T^2)$ space: for each $I \in \mathbf{Seg}(\mathbb{R})$, there is an associated interval $(s_I, t_I)$ (possibly empty) such that $(I, t) \in \supp \bbv$ if and only if $t \in (s_I, t_I)$. It suffices to just store this tuple of $(s_I, t_I)$: that is, for each time interval record the smallest and largest spatial scale for which it appears in $\mathbf{Dyn}$. 
            
            Analogously $\mathbb{R}^d$ can approximated by the discrete sublattice $[n]^d$ and the support of a discretized $\mathbb{R}^d$-module can be stored with $\mathcal{O}(n^{d-1})$ space: for each $x \in [n]^{d-1}$, by convexity it suffices to store the smallest and largest values of the $d$th coordinate that lie in the support at $x$.
            \begin{theorem}\label{dEquadratic}
                There exists an $\mathcal{O}(n^{d-1} d \log{n})$ algorithm to compute the erosion distance beween rank invariants of two antichain-decomposable $[n]^d$ modules. This specializes to an $\mathcal{O}((S+T)^2\log(S+T))$ algorithm to compute the erosion distance between antichain-decomposable $\mathbf{Dyn}$ modules.
            \end{theorem}

            We will now detail this algorithm. This proceeds by sweep-line, for which the following definitions will be necessary.
            \begin{definition}
                For the discrete lattice $[n]^d$ under the product order, the diagonal projection is the map $\pi: [n]^d \rightarrow \mathbb{Z}^{d-1}$ given by \[(x_1, \ldots, x_{d-1}, x_d) \mapsto (x_d - x_1, \ldots, x_d - x_{d-1}).\]
                Note that each point in the base space $\pi([n]^d)$ uniquely identifies a line in $[n]^d$. For a $v \in \pi([n]^d)$, we will refer to $\pi^{-1}(v)$ as the fiber of $v$.
            \end{definition}

            With this, we can now construct the algorithm.
            \paragraph*{The Algorithm.}
                \begin{algorithmic} \label{alg:dEalgo}
                    \Function{ErosionDistance}{$\mathbb{V}$, $\mathbb{W}$} 

                    \State \( P_\mathbb{V} \gets \) \textbf{InitProjectedSweepList}($\mathbb{V}$)
                    \State \( P_\mathbb{W} \gets \) \textbf{InitProjectedSweepList}($\mathbb{W}$)
                    
                    \State \( \mathrm{ans} \gets \) \( 0 \)

                    \For{each line $L$ in $[n]^{d-1}$ parallel to $(1, \ldots, 1)$}
                        \State Initialize a set $S_\mathbb{V}$ to be empty.
                        \State Initialize a set $S_\mathbb{W}$ to be empty.
                        
                        \For{$v \in L \cap \pi([n]^d)$} \Comment{Iterate in the lexicographic order on $\mathbb{Z}^{d-1}$}

                            \State $S_\mathbb{V}$.\textbf{Update}($P_\mathbb{V}$, $v$)
                            \State $S_\mathbb{W}$.\textbf{Update}($P_\mathbb{W}$, $v$)
    
                            \State  \( a \gets \) \textbf{ReconstructMin}(\( S_\mathbb{V} \), v)
                            \State \( b \gets \) \textbf{ReconstructMax}(\( S_\mathbb{V} \), v)
                            \State \( c \gets \) \textbf{ReconstructMin}(\( S_\mathbb{W} \), v)
                            \State \( d \gets \) \textbf{ReconstructMax}(\( S_\mathbb{W} \), v)
                            
                            \State \( x \gets \max (\textbf{Erosion1d}(a,\,b,\,c,\,d), \textbf{Erosion1d}(c,\,d,\,a,\,b) ) \)
                            
                            \State \( \mathrm{ans} \gets \) \( \max(\mathrm{ans},\,x) \)
                            
                        \EndFor
                    \EndFor

                    \State \Return \( \mathrm{ans} \)

                    \EndFunction

                    \State
                    \Function{InitProjectedSweepList}{$\mathbb{V}$}
                        \State Initialize $P^{\mathrm{start}}, P^{\mathrm{end}}$ to be hash maps with keys in $\pi([n]^d)$ and values being lists, by default empty.
                        \For{$x \in [n]^{d-1}$}
                            \State \( (a,\,b) \gets \) \textbf{SupportInterval}($\mathbb{V}$, $x$)
                            \State Append $x$ to $P^{\mathrm{start}}$[$\pi((x, a))$]
                            \State Append $x$ to $P^{\mathrm{end}}$[$\pi((x, b))$]
   
                        \EndFor
                        \State \Return $(P^{\mathrm{start}},\,P^{\mathrm{end}})$
                    \EndFunction
                    
                \end{algorithmic}

        This relies on a few auxillary functions, detailed below:
        \begin{itemize}
            \item \textbf{SupportInterval}($\mathbb{V}$, $x$). For an index $x \in [n]^{d-1}$ representing the first $(d-1)$ coordinates of a point in $[n]^d$, \textbf{SupportInterval}($\mathbb{V}$, $x$) returns an interval $(a, b)$ such that $(x, t) \in \mathbb{R}_\bbv^d$ for all $t \in (a, b)$. The interval is empty if $(v, x) \notin \mathbb{R}_\bbv^d$ for all $t$.
            
            \item The function $S$.\textbf{Update}($(P^{\mathrm{start}}, P^{\mathrm{end}})$, $v$) inserts into $S$ all $x$ satisfying $x \in P^{\mathrm{start}}[v]$, and removes all $x$ from $S$ satisfying $x \in P^{\mathrm{end}}[v]$.

            \item The function \textbf{ReconstructMin}$(S, v)$. For $S, v$ first find the minimum $x$ in $S$ (according to the product order on $\mathbb{Z}^{d-1}$). Then, find the unique $t$ such that $(x, t)$ lies on the fiber of $v$. Return $t$. For this to be well-defined a few invariants need to be upheld, described in the analysis.

            \item The function \textbf{ReconstructMax}$(S, v)$ performs the analogous function as \textbf{ReconstructMin}$(S, v)$, but finds the max element instead of the min.

           \item The function \textbf{Erosion1d}($a,\,b,\,c,\,d$) returns the maximum erosion in the one dimensional support $[a,\,b]\setminus[c,\,d]$ as a subset of $\mathbb{R}$.

        \end{itemize}
            \paragraph*{Analysis.}
            Here we analyze the algorithm in \Cref{alg:dEalgo} and prove its correctness and runtime guarantees.
            
            \begin{theorem}
                For any two antichain-decomposable $[n]^d$-modules $\bbv, \bbw$, the algorithm in \Cref{alg:dEalgo} computes $d_E(\bbv, \bbw)$.
            \end{theorem}
            \begin{proof}
                Note that for a thin module with convex support (in particular an ACD module), the slice of the support of $\bbv$ and $\bbw$ along a fiber $v$ is either empty or a single interval. By the reduction of the erosion distance outlined earlier, it is sufficient to prove that for each fiber $v$, the intersection of $\mathbb{R}_\bbv^d$ and $\mathbb{R}_\bbw^d$ with the fiber $v$ is computed: in the pseudocode this is $a, b, c, d$.

                Suppose we are considering a particular fiber $v$ within the innermost loop. For a point $x \in [n]^{d-1}$, let $(a_x, b_x) := \textbf{SupportInterval}(\bbv, x)$. We say $x$ is \emph{active} for fiber $v$ if the fiber $v$ intersects the segment $[(x, a_x), (x, b_x)]$.

                We claim that $S_\bbv$ is exactly the set of active points $x \in [n]^{d-1}$ for fiber $v$. For any point $x \in [n]^{d-1}$, define $A_x$ to be the set of all $v$ such that $v$ intersects $[(x, a_x), (x, b_x)]$ (i.e. the set of all $v$ where $x$ \emph{should} be active). Note that by linearity of $\pi$, we have $A_x = [\pi(x, a_x), \pi(x, b_x)]$. Because $\pi$ is order-preserving in the last coordinate when we fix the first $(d-1)$ coordinates (i.e. fix $x$), this is a well-defined segment with $\pi((x, a_x)) \leq \pi((x, b_x))$.
                
                Now, observe that $A_x$ for any fixed $x$ is a line segment parallel to $(1, \ldots, 1)$ in $[n]^{d-1}$: so by our iteration order, once a point $x$ becomes active, it remains active until the sweep-line iteration $v$ becomes equal to $\pi(x, b_x)$, which is correctly reflected in our updates to $S_\bbv$. So for any fiber $v$, $S_\bbv$ correctly maintains the set of all active points. 

                Finally, to see that \textbf{ReconstructMin} and \textbf{ReconstructMax} are well-defined, we need to show that elements of $S_\bbv$ are always comparable. Suppose $x, y \in S_\bbv$. Then by construction, there must be $t_x, t_y$ such that $(x, t_x)$ and $(y, t_y)$ both lie on the fiber $v$. Then $\pi(x, t_x) = \pi(y, t_y) = v$. Expanding $\pi$, this requires that $x$ and $y$ differ strictly by a scalar multiple of the vector $(1, \ldots, 1)$. Consequently, the active set $S_\bbv$ strictly forms a totally ordered chain under the product order, and so taking the $\min$ and $\max$ is well-defined.

                Thus, the intersection of $\mathbb{R}_\bbv^d$ and $\mathbb{R}_\bbw^d$ with the fiber $v$ is computed correctly, establishing the correctness of this algorithm.
            \end{proof}

            \begin{theorem}
                For any two antichain-decomposable $[n]^d$-modules $\bbv, \bbw$, the algorithm in \Cref{alg:dEalgo} runs in $O(n^{d-1} d \log{n})$ time.
            \end{theorem}
            \begin{proof}
                We will briefly analyze the runtime of both the initialization and sweep-line phase. The representation of $\bbv, \bbw$ is assumed to be such that \textbf{SupportInterval} is an $O(1)$ query. This can be accomplished by a $(d-1)$ dimensional array of tuples, such that for each $x \in [n]^{d-1}$ the tuples represents the range of $t$ such that $(x, t) \in \mathbb{R}_{\bbv}^d$.
                
                Under these assumptions, \textbf{InitProjectedSweepList} is an $O(n^{d-1})$ computation. The $P^{\mathrm{start}}, P^{\mathrm{end}}$ maps produced contains a list of sweep-line events: there are $O(n^{d-1})$ total events across all keys. 

                The sweep-line iteration processes $|\pi([n]^d)| = O(n^{d-1})$ total fibers. For each fiber, the operation S.\textbf{Update} can be accomplished in $O(d \log{|S|})$ time if $S$ is represented by a balanced BST ordered with the lexicographic order (note that as the elements of $S_\bbv, S_\bbw$ are always pairwise comparable in the product order, the lexicographic order gives us a total order that is ``consistent'' with the product order). As $|S_\bbv|$ and $|S_\bbw|$ are both $O(n)$, this operation is $O(d \log{n})$. The operations \textbf{ReconstructMin}, \textbf{ReconstructMax} and \textbf{Erosion1d} are all constant time, so the sweep-line iteration is $O(n^{d-1} \cdot d \log{n})$.

                And so the algorithm in \Cref{alg:dEalgo} runs in $O(n^{d-1} d \log{n})$ time.
            \end{proof}

           \paragraph{The proof of \Cref{dEquadratic}.} By combining the previous two theorems, \Cref{dEquadratic} is proven. Indeed, note that as the rank invariant of $\bdyn$-indexed modules is indexed by $\mathbb{R}_{\times}^6$ (which is $\mathbb{R}^6$ with a different ordering), the algorithm in \Cref{alg:dEalgo} can be applied to a $\bdyn$-indexed ACD module via a suitable modification of the projection function $\pi$. In particular, the following will work: \[ \pi_{\bdyn}(\mathbf{a}) = (\mathbf{a}_1 + \mathbf{a}_2, \mathbf{a}_1 + \mathbf{a}_3, \mathbf{a}_1 + \mathbf{a}_4, \mathbf{a}_1 - \mathbf{a}_5, \mathbf{a}_1 - \mathbf{a}_6) \qquad \forall \mathbf{a} \in \mathbb{R}_{\times}^6. \]
           Note that as the rank invariant is defined to be $0$ at nontrivially nonadmissible points, and $+\infty$ at trivially admissible points (see \Cref{def:bdynrankinvariant}), such points do not contribute to the supremum in the erosion distance definition.

            \begin{remark}
                An algorithm to compute an interleaving distance between rank invariants (equivalent to the erosion distance) is presented in \cite[Theorem 5.4]{spatiotemporal}. For persistence modules over the $d$-dimensional poset $[n]^d$, the rank invariant is $2d$-dimensional, and the algorithm \cite[Theorem 5.4]{spatiotemporal} for computing their erosion distance attains an expected runtime of $O(n^{2d} \log{n})$. Moreover, by \cite[Section 4]{searchcomplexity} any search-based algorithm to compute the erosion distance between such modules has expected cost at least $O(n^{2d} \log{n})$.

                Compare this to \Cref{dEquadratic}, which for convex thin persistence modules over the same poset $[n]^d$, can compute their erosion distance in $O(n^{d-1}d\log n)$ time in the worst case. 
            \end{remark}
        \subsection{Extension of $d_E$ algorithm to non-thin modules} \label{sect:dealgoextension}
            Here, we provide an extension of the algorithm in \Cref{alg:dEalgo} to a more general class of persistence modules. These modules are not assumed to be thin or antichain-decomposable, but satisfy a generalized notion of convexity.
            \begin{definition}[Rank-maximal module] \label{def:rankmaximal}
                Fix a poset $\mathcal{P}$. We say a persistence module $\bbv: \mathcal{P} \rightarrow \bvec$ is \emph{rank-maximal} if for every $p, q \in \mathcal{P}$ with $p \leq q$, the following holds: 
                \[\mathrm{rank}\, \bbv(p \leq q) = \min{(\dim \bbv(p), \dim \bbv(q))}.\]
                Note that this implies that every structure map is either injective or surjective (or possibly both).
            \end{definition}
            To motivate the study of rank-maximal modules, we first recall the following.
            \begin{definition}[Simplicial filtration]
                A simplicial filtration over a poset $\mathcal{P}$ is a functor $K: \mathcal{P} \rightarrow \bsimp$ such that for any $p, q \in \mathcal{P}$ with $p \leq q$, we have $K(p) \subset K(q)$.
            \end{definition}
            Indeed, rank-maximal modules arise naturally from clustering information of a fixed set of points.
            \begin{proposition}
                 Fix a poset $\mathcal{P}$. Let $K: \mathcal{P} \rightarrow \bsimp$ be a simplicial filtration that does not add points: that is, for any $p \leq q$ the zero-skeletons of $K$ are equal, i.e. $K_0(p) = K_0(q)$.
                
                Define the module $\bbv: \mathcal{P} \rightarrow \bvec$ by $\bbv := H_0 \circ K$, where $H_0$ is the zeroth homology functor. Then $\bbv$ is a rank-maximal module.
            \end{proposition}
            In particular, this shows that the $k=0$ spatiotemporal persistent homology \cite{spatiotemporal} \emph{without} curvature set simplification yields rank-maximal modules, and thus can benefit from the algorithm presented in this section.

            \begin{theorem}
                Suppose $\bbv, \bbw: [n]^d \rightarrow \bvec$ are rank-maximal modules. Let \[M = \max{\Big(\max_{p \in [n]^d} \dim \bbv(p), \max_{p \in [n]^d} \dim \bbw(p)\Big)}.\]
                Then $d_E(\mathrm{rk}_\bbv, \mathrm{rk}_\bbw)$ can be computed in $O(M \cdot n^{d-1} d \log{n})$ time.
            \end{theorem}
            \begin{proof}
                For $i \in \{0, \ldots, M\}$ define the superlevel set $\supp^{(i)} \bbv$ and superlevel module $\bbv^{(i)}$ by \[\supp^{(i)} \bbv := \{p \in [n]^d : \dim \bbv(p) \geq i \} \qquad \bbv^{(i)} := \mathbb{I}_{\supp^{(i)} \bbv} \] and similarly for $\bbw$. Note that by rank-maximality, $\bbv^{(i)}$ and $\bbw^{(i)}$ are now thin convex modules. The thinness is immediate, and to show convexity, suppose $p, q, r \in \supp^{(i)} \bbv$ with $p \leq r \leq q$. Then the map $\bbv(p \leq q)$ factors through $\bbv(r)$, and by rank-maximality we thus have \[i \leq \min{(\dim \bbv(p), \dim \bbv(q))} = \mathrm{rank}\, \bbv(p \leq q) \leq \mathrm{rank}\, \bbv(p \leq r) \leq \dim \bbv(r),\]
                and so $r \in \supp^{(i)} \bbv$.
                
                Additionally, for any $\varepsilon > 0$ define the erosion of the rank invariant $\mathrm{rk}_{\bbv}^{+\varepsilon}$ by $\mathrm{rk}_{\bbv}^{+\varepsilon}([p, q]) := \mathrm{rk}_\bbv([p^{-\varepsilon}, q^{+\varepsilon}])$ for any segment $[p, q]$ and similarly for $\mathrm{rk}_{\bbw}^{+\varepsilon}$. It is clear that $d_E(\mathrm{rk}_\bbv, \mathrm{rk}_\bbw)$ is the infimum of all $\varepsilon$ such that $\mathrm{rk}_{\bbv}^{+\varepsilon} \leq \mathrm{rk}_{\bbw}$ and $\mathrm{rk}_{\bbw}^{+\varepsilon} \leq \mathrm{rk}_{\bbv}$ (we will refer to any such $\varepsilon$ as a feasible erosion).

                Now, for any segment $I = [p, q]$ the following holds by rank-maximality: \[\mathrm{rk}_{\bbv}(I) = \min{(\dim \bbv(p), \dim \bbv(q))} = \sum_{i=1}^M \mathrm{rk}_{\bbv^{(i)}}(I),\]
                where the sum is taken pointwisely. And so we have \[\mathrm{rk}_{\bbv} = \sum_{i=1}^M \mathrm{rk}_{\bbv^{(i)}} \qquad \mathrm{rk}_{\bbv}^{+\varepsilon} = \sum_{i=1}^M \mathrm{rk}^{+\varepsilon}_{\bbv^{(i)}} \qquad \mathrm{rk}_{\bbw} = \sum_{i=1}^M \mathrm{rk}_{\bbw^{(i)}} \qquad \mathrm{rk}^{+\varepsilon}_{\bbw} = \sum_{i=1}^M \mathrm{rk}^{+\varepsilon}_{\bbw^{(i)}}.\]                
                Combining these properties, we can see that if $\varepsilon$ is a feasible erosion for the pairs $\mathrm{rk}_{\bbv^{(i)}}, \mathrm{rk}_{\bbw^{(i)}}$ for all $i$, then it must be a feasible erosion for the pair $\mathrm{rk}_\bbv, \mathrm{rk}_\bbw$ by the levelset additivity shown above. And clearly if $\varepsilon$ is a feasible erosion for $\mathrm{rk}_\bbv, \mathrm{rk}_\bbw$ then it must also be feasible for each individual $\mathrm{rk}_{\bbv^{(i)}}, \mathrm{rk}_{\bbw^{(i)}}$.

                Thus, we have shown that \[d_E(\mathrm{rk}_\bbv, \mathrm{rk}_\bbw) = \max_{i \in \{1, \ldots, M\}} \beta_i(\bbv, \bbw) \qquad \beta_i(\bbv, \bbw) := d_E(\mathrm{rk}_{\bbv^{(i)}}, \mathrm{rk}_{\bbw^{(i)} }).\]
                Each $\beta_i(\bbv, \bbw)$ can be computed in $O(n^{d-1} d \log{n})$ time by the algorithm in \Cref{alg:dEalgo}: in particular, note that we do not need to physically construct $\bbv^{(i)}$ and $\bbw^{(i)}$ in memory, as all we need is a representation of $\bbv, \bbw$ such that $\mathbf{SupportInterval}(\bbv^{(i)}, v)$ is an $O(1)$ operation for all $i, v$. We need to perform $M$ such calculations, and so this final complexity is $O(M \cdot n^{d-1} d \log{n})$, which is the desired result.
            \end{proof}
            
\section{Experimental Results} \label{sect:experimental}
    To validate the implementability and distinguishing power of our method, we analyze dynamic metric spaces produced by the Boids model \cite{boidsmodel}. A repository with code supporting this experiment is available at \cite{dyncurvrepo}, including both computation of distances and fixed-seed pseudorandom generation of boids.
    
    The particular variant of the Boids model utilized simulates a finite set of points moving on the flat $2$-torus $\mathbb{T}^2$, with the following parameters governing the behavior of each point:
    
    \begin{itemize}
        \item Cohesion: The tendency for a boid to be attracted to neighboring boids.
        \item Separation: The tendency for a boid to steer away from neighbors that are very close. 
        \item Alignment: The tendency for neighboring boids to match velocity vectors.
    \end{itemize}
    
    The parameters were varied to produce $5$ distinct behaviors, and $10$ randomly generated flocks were selected per behavior. Each flock was run for $t=600$ timesteps, resulting in a total of $50$ dynamic metric spaces.

    We then computed the persistence set invariant for each flock, as per \Cref{def:dyncurvpmod}. For each flock, a total of $15,000$ $(2k+2)$-size subsets for $k \in \{0, 1\}$ were selected and their spatiotemporal persistent homology modules were computed: this produces the approximate persistence set of the DMS. An approximate k-center algorithm (Gonzalez's algorithm) was used with the $d_E$ metric on the computed modules to select $2,000$ ``representative'' samples.
    
    To lower bound the distance $d_{\mathrm{dyn}}$ between two flocks, we computed instead the metric $d_H^{\mathrm{ero}}$ (\Cref{def:persistenceseterosion}) between the persistence sets associated with these two flocks, for each $k$. The computation of the Hausdorff distance involved in the definition of $d_H^{\mathrm{ero}}$ requires all pairwise erosion distances between modules to be calculated, but this is tractable by parallel computation.
    
    This procedure was conducted for the $k=0$ and $k=1$ degree homology, and the final distance matrix between flocks is the elementwise maximum of the $k=0$ and $k=1$ distance matrices.

    Using both Single Linkage Hierarchical Clustering (\Cref{fig:sub1}) and Classical Multidimensional Scaling (\Cref{fig:sub2}), this distance matrix was visualized. The results demonstrate good separation between behaviors (\Cref{fig:plotsboth}).

    \begin{figure}[bpth]
        \begin{subfigure}[b]{0.47\textwidth}
            \centering
            \includegraphics[width=\textwidth]{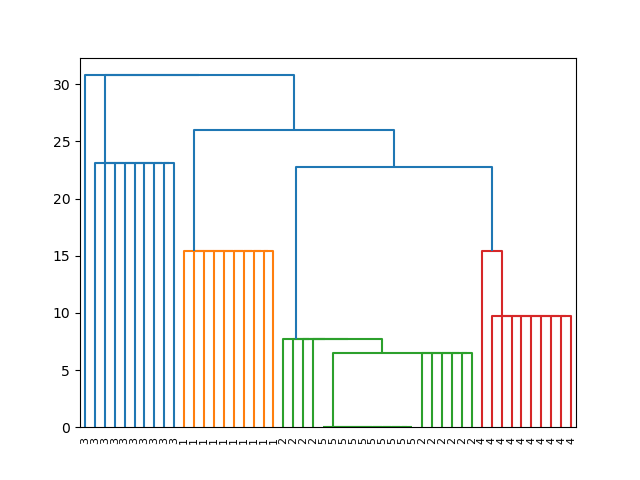}
            \caption{Single linkage dendrogram plot of distance matrix between flocks induced by $d_E$. Note that the distance between any pair of points for behavior $5$ is $0$.}
            \label{fig:sub1}
        \end{subfigure}
        \hfill
        \begin{subfigure}[b]{0.47\textwidth}
            \centering
            \includegraphics[width=\textwidth]{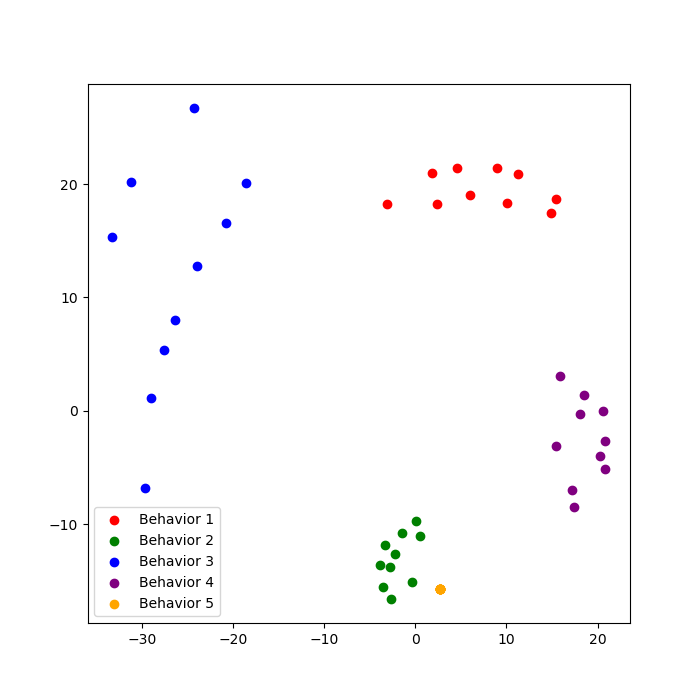}
            \caption{Classical MDS plot of distance matrix between flocks induced by $d_E$. Note that all $10$ points for behavior $5$ are coincident.}
            \label{fig:sub2}
        \end{subfigure}
        \caption{Visualization of distance matrix between flocks produced by $d_E$-induced distance}
        \label{fig:plotsboth}
    \end{figure}
    
    Additionally, to assess the quality of clustering, we computed the average error rate of the $1$-nearest neighbor classifier. For each behavior, we choose one random flock as a ``representative", and classify all remaining flocks based on the $1$-nearest neighbor classifier produced by these representatives. This procedure is repeated $10^5$ times, and the misclassification rate of each trial is averaged to yield a final $1$-nearest neighbors misclassification rate. As a baseline, we compare against the results of the PHoDMSs repository (\cite{phodmsrepo}), which implements the algorithms described in \cite{spatiotemporal}. The corresponding $1$-NN accuracies are:
    \begin{itemize}
        \item $\mathrm{NN}_1(d_E) = 98.57\%$
        \item $\mathrm{NN}_1(\mathrm{PHoDMSs}) = 72.23\%$
    \end{itemize}
    These experiments were run on an Intel(R) Xeon(R) Gold 6338 CPU machine with $32$ gigabytes of memory. The $d_E$ stage of the experiment took $63$ minutes wall time, and the PHoDMSs baseline stage took $31$ hours.

\section{Proof of Stability Inequality} \label{sec:stabilityproof}
    In this section we establish the stability of our construction, leveraging the result \cite[Theorem 4.1]{spatiotemporal}. For a dynamic metric space $\gamma_X = (X, d_X)$ and a subset $X' \subseteq X$, let $\gamma_{X'}$ denote the subspace $\gamma_{X'} = (X', d_X |_{X'})$.
    \begin{restatable}[]{lemma}{dynghsubsetrestate} \label{lem:dynghsubset}
        Suppose $\gamma_X = (X, d_X)$ and $\gamma_Y = (Y, d_Y)$ are two dynamic metric spaces, and let $\delta > 0$ be arbitrary. Then for all $X' \subseteq X$, there exists a $Y' \subseteq Y$ such that $|Y'| \leq |X'|$, and $d_{\text{dyn}}(\gamma_{X'}, \gamma_{Y'}) \leq d_{\text{dyn}}(\gamma_X, \gamma_Y) + \delta$
    \end{restatable}
    For a proof of this lemma, see \Cref{app:dynghsubsetproof}. We can now show:
    \stabrestate*
    \begin{proof}
        For notational convenience, define $P = D_k^n(\gamma_X)$ and $Q = D_k^n(\gamma_Y)$.
        
        Let $\delta > 0$ be arbitrary, and $\bbv \in P$. By \Cref{def:dyncurvpmod}, there is a corresponding $S \subseteq X$ and $\gamma_S \in K_n(\gamma_X)$ such that $\bbv = H_k(\mathcal{R}^{\text{lev}}(\gamma_S))$. By \Cref{lem:dynghsubset}, we can choose a $T \subseteq Y$ and a $\gamma_T \in K_n(\gamma_Y)$ such that
        \[d_{\text{dyn}}(\gamma_S, \gamma_T) \leq d_{\text{dyn}}(\gamma_X, \gamma_Y) + \delta.\]
        Define $\bbw := H_k(\mathcal{R}^{\text{lev}}(\gamma_T))$, and by \Cref{def:dyncurvpmod} it follows that $\bbw \in Q$. By \cite[Theorem 4.1]{spatiotemporal}, we have the inequality $d_I^{\bvec}(\bbv, \bbw) \leq 2 \cdot d_{\text{dyn}}(\gamma_S, \gamma_T)$,
        and thus for any $\bbv \in P$ we have found a corresponding $\bbw \in Q$ such that
        \[d_I^{\bvec}(\bbv, \bbw) \leq 2 \cdot d_{\text{dyn}}(\gamma_S, \gamma_T) \leq 2 \cdot d_{\text{dyn}}(\gamma_X, \gamma_Y) + 2 \delta.\]
        As $\bbv$ and $\delta > 0$ are arbitrary, this yields 
        \[\max_{\bbv \in P} \min_{\bbw \in Q} d_I^{\bvec}(\bbv, \bbw) \leq 2 d_{\text{dyn}}(\gamma_X, \gamma_Y).\]
        By interchanging $\gamma_X$ and $\gamma_Y$, we have the symmetric statement
        \[\max_{\bbv \in Q} \min_{\bbw \in P} d_I^{\bvec}(\bbv, \bbw) \leq 2 d_{\text{dyn}}(\gamma_X, \gamma_Y),\]
        which proves that $d_H\left(D_k^n(\gamma_X), D_k^n(\gamma_Y)\right) \leq 2 \cdot d_{\text{dyn}}(\gamma_X, \gamma_Y)$.
    \end{proof}

\section{Conclusion}
    We introduce a construction for analysis of dynamic data that contends with the computational challenges arising from multidimensional persistence by working with many sufficiently small subsets of the original data. This produces interval-decomposable modules that admit an efficient algorithm for computing metrics such as the erosion distance, which can be applied to arbitrary convex thin modules, independent of our construction.

    There are several avenues to improve upon these results. In particular, an interesting question regarding the modules we construct is the existence of a ``synthetic'' definition of such modules. We already have a necessary condition for a module $\bbv$ to be the result of our construction: it must be antichain-decomposable. But this is not sufficient: in a dynamic metric space $(X, d_X)$, the distance function between any two points $x, y \in X$ through time must be continuous, which imposes additional requirements on the structure of such modules. Additionally, combinatorial properties of the Rips filtration will impose further constraints. Under what conditions can a $\bdyn$-indexed persistence module be realized by taking the Rips spatiotemporal persistent homology of a $(2k+2)$-DMS?

    Additionally, the practical handling of such modules has room to be strengthened. Currently, the support of all modules is stored as a potentially dense matrix. Is there a way to efficiently ``sketch'' the support, i.e. sparsify the support or approximate it in terms of simpler primitives? This will allow for significantly more efficient storage and computation. Furthermore, the Hausdorff distance between persistence sets is currently calculated by computing the pairwise erosion distance between all modules. More elaborate data structures may significantly reduce the number of computations, leading to practical performance gains.

    Finally, it is natural to ask whether our framework can be leveraged to detect and characterize motifs in dynamic metric spaces, i.e. recurrent patterns or behaviors exhibited by many sub-DMSs of a larger DMS. In particular, interpretation of $(2k+2)$ persistence sets from a perspective of motifs, and efficient storage and clustering of such persistence set elements, is a subject of future work. This is analogous to motif discovery in graphs and networks, introduced in \cite{networkmotifs}.

\renewcommand\refname{Bibliography}
\bibliography{main.bib}

\begin{appendices}
    \section{The $\lambda$-Slack Interleaving Distance} \label{app:lambdaslackdistance}
        We present the $\lambda$-slack interleaving distance, a generalization of the Gromov-Hausdorff distance to the dynamic setting, introduced in \cite{stablesignatures}. These are a family of metrics on the space of dynamic metric spaces (modulo strong isomorphism), parameterized by a number $\lambda \in [0, \infty)$.
    
        \begin{definition}[Tripods \cite{stablesignatures}]
            For two sets $X$ and $Y$, a tripod is a set $Z$ and two surjective maps $\varphi_X: Z \rightarrow X$ and $\varphi_Y: Z \rightarrow Y$. We denote a tripod $R$ by $X \overset{\varphi_X}{\longtwoheadleftarrow} Z \overset{\varphi_Y}{\longtwoheadrightarrow} Y$
        \end{definition}
    
        \begin{definition}[Comparison of functions by tripods \cite{stablesignatures}]
            Let $X$ and $Y$ be any two sets, with maps $d_1: X \times X \rightarrow \mathbb{R}_{+}$ and $d_2: Y \times Y \rightarrow \mathbb{R}_{+}$. Given a tripod $R: X \overset{\varphi_X}{\longtwoheadleftarrow} Z \overset{\varphi_Y}{\longtwoheadrightarrow} Y$, we say $d_1 \leq_{R} d_2$ if and only if \[\varphi_X^{*} d_1(z, z') \leq \varphi_Y^{*} d_2(z, z') \qquad \forall z, z' \in Z\]. Here, $\varphi_X^{*} d_1(z, z')$ denotes the pullback of $d_1$, i.e. $\varphi_X^{*} d_1(z, z') = d_1(\varphi_X(z), \varphi_X(z'))$
        \end{definition}
        For any $t \in \mathbb{R}$, $\varepsilon > 0$ we let $[t]^{+\varepsilon} := [t - \varepsilon, t + \varepsilon]$. Additionally, we will define
        \begin{definition}[\cite{stablesignatures}]
            Let $X$ be a set and $d_X: X \times X \rightarrow \mathbb{R}_+$. Then for any $t \in \mathbb{R}$ and $\varepsilon > 0$, define $\bigvee_{[t]^{+\varepsilon}} d_X$ by \[\left(\bigvee_{[t]^{+\varepsilon}} d_X\right)(x, x') := \min_{t' \in [t]^{+\varepsilon}} d_X(t')(x, x') \qquad \forall x, x' \in X.\]
        \end{definition}
        With this definition, we can now introduce the $\lambda$-slack distortion, which is used to construct the $\lambda$-slack interleaving distance.
        \begin{definition}[$\lambda$-distortion of a tripod \cite{stablesignatures}]
            Suppose $\gamma_X = (X, d_X)$ and $\gamma_Y = (Y, d_Y)$ are two dynamic metric spaces. Let $R: X \overset{\varphi_X}{\longtwoheadleftarrow} Z \overset{\varphi_Y}{\longtwoheadrightarrow} Y$ be a tripod. We say $R$ is a $(\lambda, \varepsilon)$-tripod if
            \begin{center}
                For all $t \in \mathbb{R}$, $\bigvee_{[t]^{+\varepsilon}} d_X \leq_{R} d_Y(t) + \lambda \cdot \varepsilon$ and $\bigvee_{[t]^{+\varepsilon}} d_Y \leq_{R} d_X(t) + \lambda \cdot \varepsilon$
            \end{center}
            The $\lambda$-distortion $\mathrm{dis}_\lambda^{\mathrm{dyn}}(R)$ is the infimum of all $\varepsilon > 0$ such that $R$ is a $(\lambda, \varepsilon)$-tripod.
        \end{definition}
    
        \begin{definition}[$\lambda$-slack interleaving distance between DMSs \cite{stablesignatures}] \label{def:slackinterleaving}
            For any $\lambda \geq 0$ and dynamic metric spaces $\gamma_X, \gamma_Y$, define the $\lambda$-slack interleaving distance as
            \[d_{\mathrm{dyn}, \lambda}(\gamma_X, \gamma_Y) = \min_{R} \text{dis}_\lambda^{\text{dyn}}(R)\]
            Where we minimize over all tripods $R: X \overset{\varphi_X}{\longtwoheadleftarrow} Z \overset{\varphi_Y}{\longtwoheadrightarrow} Y$.
        \end{definition}

        For brevity, we make the following definition:
        
        \begin{definition}[$d_{\mathrm{dyn}}$ between DMSs \cite{stablesignatures}]\label{def:ddynmetric}
            We will refer to the $\lambda=2$ slack interleaving distance $d_{\mathrm{dyn}, 2}$ as $d_{\mathrm{dyn}}$.
        \end{definition}
    
    \section{Proof of \Cref{lem:dynghsubset}} \label{app:dynghsubsetproof}
        We will show that:
        \dynghsubsetrestate*
        \begin{proof}
            For notational convenience, let $\Delta = d_{\text{dyn}}(\gamma_X, \gamma_Y)$.
            
            By \Cref{def:slackinterleaving}, there is a dynamic metric space $(Z, d_Z)$ and a surjective tripod $R$ given by $X \overset{\varphi_X}{\longtwoheadleftarrow} Z \overset{\varphi_Y}{\longtwoheadrightarrow} Y$, such that $\text{dis}^{\text{dyn}}(R) < \Delta + \delta$.
            
            By surjectivity of $R$, for each $x \in X'$ there is a corresponding $z_x$ such that $\varphi_X(z_x) = x$. 
            
            Let $Z' = \{z_x : x \in X' \}$, and define $Y' = \varphi_Y(Z')$.
            
            This yields a surjective tripod $R'$, denoted by $X' \overset{\varphi_X |_{Z'}}{\longtwoheadleftarrow} Z' \overset{\varphi_Y |_{Z'}}{\longtwoheadrightarrow} Y'$. Consider the dynamic metric spaces $\gamma_{X'}, \gamma_{Y'}, \gamma_{Z'}$. We claim that $\text{dis}^{\text{dyn}}(R) < \Delta + \delta$. Suppose not: then by \Cref{def:slackinterleaving}, $R'$ is not an $\varepsilon$-tripod for any $\varepsilon < \Delta + \delta$. So there must exist some $t \in \mathbb{R}$ and a pair $z, z' \in Z'$ such that one of the following holds:
            \begin{align*}
                \left(\varphi_{X'}^{*} \, d_{X'}([t]^{+\varepsilon})\right)(z, z') &> \left(\varphi_{Y'}^{*} \, d_{Y'}([t]^{+\varepsilon})\right)(z, z') + 2 \varepsilon \text{, or}\\
                \left(\varphi_{Y'}^{*} \, d_{Y'}([t]^{+\varepsilon})\right)(z, z') &> \left(\varphi_{X'}^{*} \, d_{X'}([t]^{+\varepsilon})\right)(z, z') + 2 \varepsilon.
            \end{align*}
            But because $Z' \subseteq Z$, the same $t$ and $z, z' \in Z$ shows that $R$ is not an $\varepsilon$-tripod. But this contradicts the fact that $\text{dis}^{\text{dyn}}(R) < \Delta + \delta$.
    
            Therefore, we have a tripod $R'$ such that $\text{dis}^{\text{dyn}}(R') < \Delta + \delta$, and thus our chosen $Y'$ satisfies $d_{\text{dyn}}(\gamma_{X'}, \gamma_{Y'}) \leq d_{\text{dyn}}(\gamma_X, \gamma_Y) + \delta$.
        \end{proof}

    \section{Statement of necessary results from \cite{curvaturesets}} \label{app:curvaturesetstheorem}
        Although these results are originally stated for metric spaces, they also hold in the semimetric setting. Indeed, curvature sets can be defined on an even more general setting of networks, see \cite{chowdhury2023distances}.
        \begin{theorem}[{\cite[Theorem 4.4(B)]{curvaturesets}}]
            Let $(X, d_X)$ be a metric space with $n$ points. Suppose that $n$ is even and $k \in \mathbb{Z}_+$ is such that $n = 2k+2$. Let $\mathrm{dgm}_k^{\mathrm{VR}}(X)$ denote the degree $k$ persistence diagram of the Rips filtration of $X$. 
            
            Then $\mathrm{dgm}_k^{\mathrm{VR}}(X)$ consists of a single point $(t_b(X), t_d(X))$ if and only if $t_b(X) < t_d(X)$ and is empty otherwise, where $t_b(X), t_d(X)$ are functions defined in \cite[Definition 4.1]{curvaturesets}.
        \end{theorem}
        In fact, \cite{curvaturesets} provides something stronger:
        \begin{proposition}[{\cite[Proposition 4.3]{curvaturesets}}]
            Let $(X, d_X)$ be a metric space with $n$ points. Suppose that $n$ is even and $k \in \mathbb{Z}_+$ is such that $n = 2k+2$. 

            If there exists a $r \in \mathbb{R}_+$ such that $\dim H_k(\mathcal{R}_r(d_X)) = 1$, then $\mathcal{R}_r(d_X)$ must be isomorphic, as a simplicial complex, to the cross-polytope $\mathfrak{B}_{k+1}$ (see \Cref{fig:crosspolytopes}).
        \end{proposition} 
        \input{cross_polytopes}
\end{appendices}
\end{document}

%% file: preamble.tex
\def\bbv{\mathbb{V}}
\def\bbw{\mathbb{W}}
\def\bdyn{\mathbf{Dyn}}
\def\bvec{\mathbf{Vec}}
\def\bsimp{\mathbf{Simp}}

\DeclareMathOperator{\supp}{supp}

\newcommand{\defeq}{\vcentcolon=}

\DeclareRobustCommand\longtwoheadrightarrow
     {\relbar\joinrel\twoheadrightarrow}
\DeclareRobustCommand\longtwoheadleftarrow
     {\twoheadleftarrow\joinrel\relbar}

\usepackage{tikz}
\usepackage{tikz-3dplot}
\usetikzlibrary{calc}

%% file: four_point_spaces.tex
\begin{figure}[htbp]
    \centering
    \begin{subfigure}[b]{0.47\textwidth}
        \centering
        \begin{tikzpicture}[scale=3.2]
        \coordinate (A) at (-0.5,-0.5);
        \coordinate (B) at ( 0.5,-0.5);
        \coordinate (C) at ( 0.5, 0.5);
        \coordinate (D) at (-0.5, 0.5);
        \draw[orange, dashed, line width=1.0pt] (A) -- (B) node[midway,below,yshift=-1pt,orange] {$1$};
        \draw[orange, dashed, line width=1.0pt] (B) -- (C) node[midway,right,xshift=2pt,orange] {$1$};
        \draw[orange, dashed, line width=1.0pt] (C) -- (D) node[midway,above,yshift=1pt,orange] {$1$};
        \draw[orange, dashed, line width=1.0pt] (D) -- (A) node[midway,left,xshift=-2pt,orange] {$1$};
        \fill (A) circle (1pt);
        \fill (B) circle (1pt);
        \fill (C) circle (1pt);
        \fill (D) circle (1pt);
        \end{tikzpicture}
        \caption{A space $(X, d_X)$ where $H_1(\mathcal{R}_\delta(d_X))$ for $\delta \in [1, \sqrt{2})$.}
        \label{fig:fourpointsquare}
    \end{subfigure}
    \hfill
    \begin{subfigure}[b]{0.47\textwidth}
        \centering
        \begin{tikzpicture}[scale=3.2]
        \usetikzlibrary{calc}
        \pgfmathsetmacro{\hmaj}{sqrt(3)/2}
        \pgfmathsetmacro{\hmin}{1/2}
        \coordinate (R) at (\hmaj,0);
        \coordinate (T) at (0,\hmin);
        \coordinate (L) at (-\hmaj,0);
        \coordinate (B) at (0,-\hmin);
        \draw[orange, dashed, line width=1.0pt] (T) -- (R) node[midway, right, xshift=2pt, orange] {$1$};
        \draw[orange, dashed, line width=1.0pt] (R) -- (B) node[midway, below, yshift=-1pt, orange] {$1$};
        \draw[orange, dashed, line width=1.0pt] (B) -- (L) node[midway, left, xshift=-2pt, orange] {$1$};
        \draw[orange, dashed, line width=1.0pt] (L) -- (T) node[midway, above, yshift=1pt, orange] {$1$};
        \draw[orange, dashed, line width=1.0pt] (L) -- (R);
        \draw[orange, dashed, line width=1.0pt] (B) -- (T);
        \fill (R) circle (1pt);
        \fill (T) circle (1pt);
        \fill (L) circle (1pt);
        \fill (B) circle (1pt);
        \node[orange, font=\small, fill=white, inner sep=1pt] at ($ (0,0)!0.5!(R) $) [yshift=2pt] {$\tfrac{\sqrt{3}}{2}$};
        \node[orange, font=\small, fill=white, inner sep=1pt] at ($ (0,0)!0.5!(T) $) [xshift=3pt] {$\tfrac{1}{2}$};
        \end{tikzpicture}
        \caption{A space $(Y, d_Y)$ where $H_1(\mathcal{R}_\delta(d_Y))$ is trivial for all $\delta \geq 0$.}
        \label{fig:fourpointrhombus}
    \end{subfigure}
    \caption{Examples of $4$-point metric spaces.}
    \label{fig:fourpointmetricspaces}
\end{figure}

%% file: cross_polytopes.tex
\begin{figure}[htbp]
  \centering
  \begin{subfigure}[b]{0.31\textwidth}
    \centering
    \begin{tikzpicture}[scale=2.2]
      \coordinate (u1) at (-0.75,0);
      \coordinate (u2) at ( 0.75,0);
      \draw[line width=0.9pt] (u1) -- (u2);
      \foreach \p in {u1,u2}
        \fill (\p) circle (1.2pt);
    \end{tikzpicture}
    \caption{The cross-polytope $\mathfrak{B}_1$}
    \label{fig:crosspoly1d}
  \end{subfigure}
  \hfill
  \begin{subfigure}[b]{0.31\textwidth}
    \centering
    \begin{tikzpicture}[scale=2.15]
      \coordinate (v1) at ( 1, 0);
      \coordinate (v2) at ( 0,-1);
      \coordinate (v3) at (-1, 0);
      \coordinate (v4) at ( 0, 1);
      \draw[line width=0.9pt] (v1) -- (v2) -- (v3) -- (v4) -- cycle;
      \foreach \p in {v1,v2,v3,v4}
        \fill (\p) circle (1pt);
    \end{tikzpicture}
    \caption{The cross-polytope $\mathfrak{B}_2$}
    \label{fig:crosspoly2d}
  \end{subfigure}
  \hfill
  \begin{subfigure}[b]{0.31\textwidth}
    \centering
    \tdplotsetmaincoords{80}{25}
    \begin{tikzpicture}[scale=2.2,tdplot_main_coords]
      \coordinate (P1) at ( 1, 0, 0);
      \coordinate (P2) at (-1, 0, 0);
      \coordinate (P3) at ( 0, 1, 0);
      \coordinate (P4) at ( 0,-1, 0);
      \coordinate (P5) at ( 0, 0, 1);
      \coordinate (P6) at ( 0, 0,-1);

      \draw[line width=0.9pt] (P1) -- (P3) -- (P2) -- (P4) -- (P1);
      \draw[line width=0.9pt] (P1) -- (P5) -- (P2);
      \draw[line width=0.9pt] (P1) -- (P6) -- (P2);
      \draw[line width=0.9pt] (P3) -- (P5) -- (P4);
      \draw[line width=0.9pt] (P3) -- (P6) -- (P4);

      \foreach \p in {P1,P2,P3,P4,P5,P6}
        \fill (\p) circle (1pt);
    \end{tikzpicture}
    \caption{The cross-polytope $\mathfrak{B}_3$}
    \label{fig:crosspoly3d}
  \end{subfigure}
  \caption{Some cross-polytopes}
  \label{fig:crosspolytopes}
\end{figure}
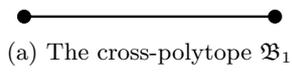
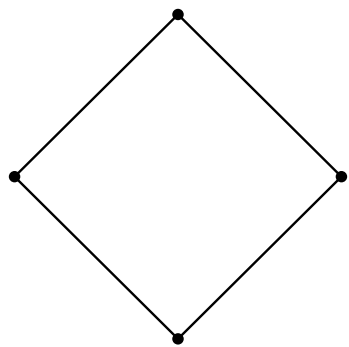
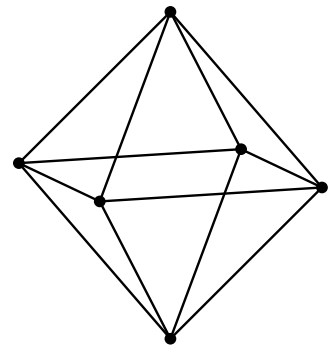